\input xy
\xyoption{all}

\documentclass[a4paper,11pt,reqno,twoside]{amsart}

\usepackage{amsmath, amscd, amsfonts, amssymb, amsthm, latexsym, url, color}
\usepackage[pdftex]{graphicx}
\usepackage{subfigure}
\usepackage{color}
\usepackage[dvipsnames]{xcolor}
\usepackage{verbatim}
\definecolor{Mycolor1}{HTML}{2B2E02} 
\definecolor{Mycolor2}{HTML}{0A2A4E}
\definecolor{Mycolor3}{HTML}{100000}
\definecolor{Mycolor4}{HTML}{349082}
\usepackage{enumerate}
\usepackage{ifpdf}
\usepackage{array}
\usepackage[T1]{fontenc}
\usepackage[latin1]{inputenc}
\usepackage{mathtools}
\usepackage{amsthm,amssymb}
\usepackage{mathptmx}
\usepackage{showlabels}

\usepackage{relsize}

\usepackage{url}
\usepackage[colorlinks,breaklinks=true]{hyperref}
\usepackage[figure,table]{hypcap}
\hypersetup{
	bookmarksnumbered,
	pdfstartview={FitH},
	citecolor={Mycolor1},
	linkcolor={Mycolor2},
	urlcolor={Mycolor2},
	pdfpagemode={UseOutlines}
}
\makeatletter
\newcommand\org@hypertarget{}
\let\org@hypertarget\hypertarget
\renewcommand\hypertarget[2]{
  \Hy@raisedlink{\org@hypertarget{#1}{}}#2
} 
\makeatother

\renewcommand*{\geq}{\geqslant}

\renewcommand*{\epsilon}{\varepsilon}

\let\oldIm=\Im
\renewcommand*{\Im}{\mathop{\oldIm\mkern-2mu 
m}\nolimits}
\let\oldRe=\Re
\renewcommand*{\Re}{\mathop{\oldRe\mkern-2mu e}\nolimits}

\newcommand*{\A}{\mathbb{A}}
\newcommand*{\bA}{\bold A}
\newcommand*{\cA}{\mathcal{A}}
\newcommand*{\bB}{\bold B}
  
\newcommand*{\bC}{\bold C}

 \newcommand*{\gd}{\mathfrak{d}} 
  \newcommand*{\gD}{\mathfrak{D}}

\newcommand*{\F}{\mathbb{F}}
\newcommand*{\cF}{\mathcal{F}}

\newcommand*{\gK}{\mathcal{K}}
\newcommand*{\gk}{\mathfrak{k}}
\newcommand*{\gok}{\mathfrak{K}}
\newcommand*{\cM}{\mathcal{M}}
\newcommand*{\cO}{\mathcal{O}}

\newcommand*{\Q}{\mathbb{Q}}
\newcommand*{\R}{\mathbb{R}}

\newcommand*{\Z}{\mathbb{Z}}
\newcommand*{\vZ}{\mathbb{Z}} 
\newcommand*{\sint}{\phantom{\int}^{\ell}\!\!\!\int}
\newcommand*{\lint}{\!\!\!\!\!\mathlarger{\phantom{\int}^{\ell}}\!\!\!\!\int}

\newcommand*{\To}{\longrightarrow}

\newlength{\arrow}
\makeatletter
 \newlength{\h@uteurnumerateur}
 \newlength{\h@uteurdenominateur}
\newcommand*{\quotientdroite}[2]{
  \mathchoice
  {
    \settoheight{\h@uteurnumerateur}{\ensuremath{\displaystyle{#1#2}}}
    \settoheight{\h@uteurdenominateur}{\ensuremath{\displaystyle{#1#2}}}
    \raisebox{0.5\h@uteurnumerateur}{\ensuremath{\displaystyle{#1}}}
    \mkern-5mu\diagup\mkern-4mu
    \raisebox{-0.5\h@uteurdenominateur}{\ensuremath{\displaystyle{#2}}}
  }
  {
    \settoheight{\h@uteurnumerateur}{\ensuremath{\textstyle{#1#2}}}
    \settoheight{\h@uteurdenominateur}{\ensuremath{\textstyle{#1#2}}}
  \raisebox{0.2\h@uteurnumerateur}{\ensuremath{\textstyle{#1}}}
  /
  \raisebox{-0.2\h@uteurdenominateur}{\ensuremath{\textstyle{#2}}}
  }
  {
    \settoheight{\h@uteurnumerateur}{\ensuremath{\scriptstyle{#1#2}}}
    \settoheight{\h@uteurdenominateur}{\ensuremath{\scriptstyle{#1#2}}}
  \raisebox{0.2\h@uteurnumerateur}{\ensuremath{\scriptstyle{#1}}}
  /
  \raisebox{-0.2\h@uteurdenominateur}{\ensuremath{\scriptstyle{#2}}}
  }
  {
    
\settoheight{\h@uteurnumerateur}{\ensuremath{\scriptscriptstyle{#1#2}}}
    
\settoheight{\h@uteurdenominateur}{\ensuremath{\scriptscriptstyle{#1#2}}}
  \raisebox{0.2\h@uteurnumerateur}{\ensuremath{\scriptscriptstyle{#1}}}
  /
  \raisebox{-0.2\h@uteurdenominateur}{\ensuremath{\scriptscriptstyle{#2}}}
  }
}
\newcommand*{\quotientgauche}[2]{
  \mathchoice
  {
    \settoheight{\h@uteurnumerateur}{\ensuremath{\displaystyle{#1#2}}}
    \settoheight{\h@uteurdenominateur}{\ensuremath{\displaystyle{#1#2}}}
    \raisebox{-0.5\h@uteurnumerateur}{\ensuremath{\displaystyle{#1}}}
    \mkern-3mu\diagdown\mkern-5mu
    \raisebox{0.5\h@uteurdenominateur}{\ensuremath{\displaystyle{#2}}}
  }
  {
    \settoheight{\h@uteurnumerateur}{\ensuremath{\textstyle{#1#2}}}
    \settoheight{\h@uteurdenominateur}{\ensuremath{\textstyle{#1#2}}}
  \raisebox{-0.2\h@uteurnumerateur}{\ensuremath{\textstyle{#1}}}
  \backslash
  \raisebox{0.2\h@uteurdenominateur}{\ensuremath{\textstyle{#2}}}
  }
  {
    \settoheight{\h@uteurnumerateur}{\ensuremath{\scriptstyle{#1#2}}}
    \settoheight{\h@uteurdenominateur}{\ensuremath{\scriptstyle{#1#2}}}
  \raisebox{-0.2\h@uteurnumerateur}{\ensuremath{\scriptstyle{#1}}}
  \backslash
  \raisebox{0.2\h@uteurdenominateur}{\ensuremath{\scriptstyle{#2}}}
  }
  {
    
\settoheight{\h@uteurnumerateur}{\ensuremath{\scriptscriptstyle{#1#2}}}
    
\settoheight{\h@uteurdenominateur}{\ensuremath{\scriptscriptstyle{#1#2}}}
  \raisebox{-0.2\h@uteurnumerateur}{\ensuremath{\scriptscriptstyle{#1}}}
  \backslash
  \raisebox{0.2\h@uteurdenominateur}{\ensuremath{\scriptscriptstyle{#2}}}
  }
}
\makeatother

\newtheoremstyle{erdfn}
  {}
  {}
  {\itshape}
  {}
  {\rmfamily\scshape}
  {.\ }
  { }
  {}
\newtheoremstyle{erthm}
  {}
  {}
  {\itshape}
  {}
  {\rmfamily\scshape}
  {.\ }
  { } 
  {}
\newtheoremstyle{errem}
  {}
  {}
  {}
  {}
  {\rmfamily\scshape}
  {.\ }
  { }   
  {}

\theoremstyle{erthm}

\theoremstyle{erdfn}  

\newtheorem*{thm}{Theorem}
\newtheorem*{prp}{Proposition}
\newtheorem*{crl}{Corollary}
\newtheorem*{lem}{Lemma}

\theoremstyle{errem}

\newtheorem*{erthm}{Remark} 

\numberwithin{equation}{section}

\newtheoremstyle{defstyle}
    {}
    {}
    {\normalfont}
    {}
    {\scshape}
    {.}
    { }
    {\thmnumber{{#2.}}\thmname{ #1}\thmnote{ (\normalfont #3)}}
 \theoremstyle{defstyle}

\newtheorem*{Definition}{Definition}

\makeatletter
\def\@subsubseccntformat#1{
  \expandafter\ifx\csname c@#1\endcsname\c@section\else
  \csname the#1\endcsname\quad
  \fi}
\makeatother

\title[Selective integration on higher adeles  and the Euler characteristic of surfaces]
{Selective integration on higher adeles \\ and the Euler characteristic of surfaces} 
\author[Weronika Czerniawska, Ivan Fesenko]{Weronika Czerniawska, Ivan Fesenko}

\AtBeginDocument{
\mathtoolsset{showonlyrefs,mathic = true}

\begin{abstract}
The space of two-dimensional geometric adeles of a surface  is  far from being a locally compact space and there is no translation countably additive invariant nontrivial measure on it.   
 At the same time, certain  subquotients of the adeles are direct limits of  compact subquotients or inverse limits of discrete subquotients, compatible in a special way.  Using this fact, the paper defines  a translation invariant measure and integration on certain  subquotients of the geometric adeles of surfaces. This theory  is considerably different from the theory of integration on  analytic adeles of surfaces. After revising aspects of one-dimensional theory, the paper includes a full definition of two-dimensional geometric adeles. A number  of their new topological properties are established.  The new translation invariant measure and integration on selective subquotients of the geometric adeles is used for integrals of certain functions in a  two-dimensional  method describing  the size of adelic cohomology groups of surfaces, without using standard adelic complexes.  A   formula for the Euler characteristic of the surface and its divisor in terms of integrals over subquotients of geometric adeles is proved.  Using the Euler characteristic,  a new  two-dimensional adelic intersection number  is introduced. For geometric surfaces it is a positive multiple of the standard intersection number. Several results in the previous study of geometric adeles are given new proofs. \end{abstract}

\maketitle

}

\begin{document}

\noindent {\bf 0}. This paper continues the higher adelic study of aspects of geometry of surfaces in \cite{ARR}, by lifting certain discrete or linear algebraic objects to two-dimensional adelic objects with richer structure and using their topological properties, as well as the new measure and integration on some of them.

The higher adelic study   in \cite{ARR}  for projective irreducible curves and projective smooth surfaces over perfect fields used  topological properties and self-duality of the one-dimensional and two-dimensional geometric adeles to study cohomology spaces of adelic complexes. The latter were proved to be of finite dimension. An additive adelic formula for the intersection of divisors was established. This formula easily implied the adelic Riemann--Roch theorem. 
Adelic objects for number fields and for arithmetic surfaces, studied in \cite{Ada}, \cite{AoA2},   \cite{CD}, also involve non-linear data associated to archimedean points of fibres. Hence, the approach in \cite{ARR} which works in the geometric case using adelic complexes and their linear cohomology spaces cannot be directly extended to  arithmetic schemes, due to the highly non-linear  archimedean aspects.

In order to have a universal approach which works both for geometric objects over finite fields and for arithmetic schemes, new points of view are needed. This paper proposes one of them, based on the use of topological properties of  adelic subquotients and integrals over them, as a  two-dimensional extension of some aspects of the one-dimensional approaches in the Iwasawa--Tate theory, see e.g. sect. 5 Ch. 3 of \cite{Fb}, and in Borisov's paper  \cite{Bo}. 

This is a very different theory from that in \cite{Ada}, \cite{AoA2}, but  some of  results in this paper  will be useful for further research related to the latter. 

Section 1 reviews  the one-dimensional case presented in a form suitable for its generalisation  to the two-dimensional case. One deduces the adelic Riemann--Roch theorem from adelic duality and standard properties of Fourier transform, see e.g. sect. 0 of \cite{ARR}. 
In the one-dimensional arithmetic case cohomologies $H^i(D)$ for a divisor $D$ can be infinite or not defined, however one can produce a non-zero numbers $h^i(D)$, $i=0,1$, which can be thought of as the volume/size of the cohomology group. These numbers  $h^i(D)$ are defined as integrals over an appropriate adelic subquotient  not depending on $D$ of a function  depending on the divisor $D$. For a  different but essentially equivalent presentation see \cite{Bo}.  The self-duality of the adeles implies
$$h^1(D_\alpha)=h^0(D_{\gk\alpha^{-1}}),$$
where $\gk$ is a relevant  idele such that $D_\gk$ is a canonical replete divisor.  
We establish a formula for the Euler characteristic of a replete divisor $D_\alpha$ of a global field, corresponding to an idele $\alpha$, with  corresponding function $f_\alpha$
$$
 \chi_{\Bbb A}(D_\alpha)=  \log \int_{\A}\, f_\alpha\, \mu_{\A}.$$
We also obtain a Riemann--Roch-like formula
 $$\chi_{\Bbb A}(D_\alpha)-\chi_{\Bbb A}(D_1)=\log |\alpha|.$$

The additive group of two-dimensional geometric adeles is not a locally compact topological space. 
It is fundamentally different from the additive group of two-dimensional analytic adeles which admits a non-trivial $\R(\!(x)\!)$-valued  translation invariant measure and integration, \cite{AoA2}, used in the study of the zeta function.  Still, both groups are self-dual topological groups or close to such. 

Section 2 starts with a full definition of   two-dimensional geometric adeles $\bA$ and subobjects $\bB$, $\bC$, and new results about  topological properties of their certain subquotients, extending \cite{ARR} and \cite{CD}.  In particular, it includes the proof of discreteness of global elements $K$ in the geometric adeles $\bA$ and the property $\bB\cap\bC=K$.

In section 3 we develop and use normalised translation invariant measure and  integration on certain subquotients of the two-dimensional geometric adeles, using their presentation as limits of locally compact subquotients. These measures  are  normalised in a way compatible with a distinguished geometric structure of the adeles. Then we define translation invariant measures on $\bA/\bB$, $\bB$, $\bA$, $\bB+\bC$, $\bC$.

In section 4, using section 3, we define and study the  size of adelic cohomology groups, without using standard  adelic complexes. The newly introduced adelic cohomology numbers $h^i(D)$, $i=0,1,2$, for divisors on surfaces are integrals over certain  adelic subquotients that are not dependent on divisors of a function  dependent on divisors.   Both in the geometric and arithmetic cases we derive the formula $$h^i(D)=h^{2-i}(\gK-D),$$ where $\gK$ is an adelic 
 \lq canonical divisor\rq.  
 
A new formula for the Euler characteristic of a divisor $D_\alpha$ corresponding to $\alpha\in \bB^\times$ in terms of integrals over certain subquotients of geometric adeles is obtained in 4-3:
 $$
 \chi_{\bA} (D_\alpha)=\log\int_{\bC}\, f_{\alpha} \, \mu_{\bC} - \log\int_{\bA/\bB}\, f_{\alpha, \bA/\bB}\,  \mu_{\bA/\bB},
 $$
for notation see the main text of the paper. This formula can be compared with the one-dimensional formula mentioned above.

 Using the Euler characteristic, we introduce a  {\em two-dimensional adelic intersection index} on surfaces
$$
[D_\alpha,D_\beta]:=\chi(D_1)-\chi(D_\alpha)-\chi(D_\beta)+\chi(D_{\alpha\beta}).
$$
Several first results about this index are included. 

In the geometric case, 
for a divisor $D$ whose support does not include an irreducible proper curve $y$, with its divisor $D_y$, we prove  
  $$[D,D_y]=\chi_{\Bbb A_{k(y)}}(D\vert_y)-\chi_{\Bbb A_{k(y)}}(0)=\log q\cdot \deg_{\Bbb A_{k(y)}}(D\vert_y).$$
This formula is used to show  that the intersection index is equal, up to a constant  positive multiple, to the usual intersection index. We also obtain new proofs of several results of \cite{ARR}. 

 An analogue of this formula in the arithmetic case and the relation of  the two-dimensional adelic  intersection index  to  the Arakelov intersection index are open questions; other open problems are stated in  4-6.

\bigskip

\noindent {\bf 1. One-dimensional case}

\bigskip 

\noindent {\bf 1-1}.  We start with one-dimensional adeles. Let $k$ be a global field, i.e. a number field or the function field of a smooth, projective, geometrically integral curve over a finite field $\F_q$. Let $\A$ be the adelic space associated to $k$. Let  $\A(0)$ be its subspace which is the product of non-archimedean rings of integers with the full archimedean components in the arithmetic case.  

For any non-trivial Haar measure $\mu$ of an abelian locally compact group and its element $\alpha$ its module $|\alpha|$  is $\mu(\alpha P)/\mu(P)$ for any measurable subset $P$ of non-zero measure; it does not depend on the choice of $P$ and the scaling of $\mu$. For the  completion $k_v$ of $k$ with respect to a  place $v$ denote by $|\,\,|_v$ the module of $k_v$.
Note that for complex places $v$ the module $|\,\,|_v$ is the square of the complex absolute value.  
 It is classical that their product over all places $v$ satisfies the product formula: this product  equals 1 on the diagonal image of $k^\times$. For non-archimedean places $v$ the function $-\log|\,\,|_v/\log \# \,k(v)\colon k_v^\times\to \vZ$,  where $k(v)$ is the residue field of $v$, is a surjective homomorphism and it gives the discrete valuation on $k_v$. Denote this function by  $v(\,\,)$, this should not lead to confusion. For archimedean places $v$ denote by $v(\,\,)$  the function $-\log|\,\,|_v$.

For an idele $\alpha=(\alpha_v)\in \A^\times$, define the {\em replete divisor}  
$$D_\alpha=-\sum v(\alpha_v) [v],$$ $v$ runs through all places of $k$.  

In the geometric case, we get the usual divisor.  We have $\log |\alpha|=\log q \cdot \deg D_\alpha$.

 In the arithmetic case, if $v(\alpha_v)=0$ for all archimedean  places   then $D_\alpha$ is a usual divisor. 

The homomorphisms from $\A^\times$ to the group of divisors and the group of replete divisors are surjective. From various perspectives, it is more useful to deal with the \lq covering object\rq\  $\A^\times$, the locally compact group of invertible adeles, than with the discrete  divisors group. 

\smallskip

For a  divisor $D=\sum n_v [v]$ define
 $$\A(D)=\{\beta\in\A: v(\beta)\geq-n_v\},$$ 
 where $v$ runs through all non-archimedean 
 places of $k$. In particular, we have $\A(0)=\A(D_1)$. Define similarly  $\A(D_\alpha)$ for a replete divisor $D_\alpha$ by not imposing inequality restrictions on  archimedean components. 
 Then    $\A(D_\alpha)=\alpha\A(0)$.

\smallskip
   
  The additive group of adeles $\A$ is the 
  direct  limit of $\A(D)$ with $D$ running over all divisors.

\smallskip

For a divisor $D$ one can consider the well known  adelic complex  
$$\cA(D):\qquad\qquad {k}\oplus \A(D)\To \A , \qquad (a,b)\mapsto a-b.$$
We have 
$$\begin{aligned} H^0(\cA(D))&=k\cap \A(D),\\
H^1(\cA(D))&=\A/(k+\A(D)).\end{aligned}$$
\smallskip 

In the geometric case its cohomologies are finite. In the arithmetic case its cohomologies are not necessarily finite, for example $H^0(\cA(0))$ is the ring of integers $O_k$ which is discrete in $O_k\otimes_{\Q}\R$.  Working with replete divisors and subsets $\{\beta\in\A: v(\beta)\geq-n_v\}$ for all places $v$ produces finite sets, but its complement of with respect to the pairing $ \langle\,\,,\,\,\rangle$ defined in 1-2 is zero. 

As usual, in this paper when we talk about an isomorphism between groups endowed with topologies we mean a topological isomorphism, i.e. an isomorphism which is a homeomorphism. 

\medskip

\noindent {\bf 1-2}.  
Let $\A'$ be the set of all elements of $\A$ with   zero components at archimedean places (so $\A'=\A$ in the positive characteristic case). Denote $\A(D)'=\A(D)\cap \A'$. 

Fix a nontrivial character $\psi_0$ of $\A$ vanishing on $k$; associated to it one has the  well known continuous  non-degenerate pairing 
$$\langle\,\,,\,\,\rangle=\langle\,\,,\,\,\rangle_{\psi_0}\colon\A\times\A\To S^1,\qquad \langle \alpha,\beta \rangle=\psi_0(\alpha\beta),$$
where $S^1$ is the complex unit circle. 

One easily sees, e.g. \cite{W},  that the additive group of $\A$ is isomorphic to its group of characters, that the complement of $k$  is $k$, 
that the complement $\A(D)^\perp$ of $\A(D)$  is $\A(\gok-D)'$ where $\gok=\sum - n_v [v]$, the sum over non-archimedean places $v$, with $n_v$ equal to the $v$-valuation of the different $\gd_v$ of $k_v$.

Denote by
 $\cA(D)'$  the complex 
$${k}\oplus \A(D)'\To \A, \quad (a,b)\mapsto a-b.$$ 
Then 
the character group of $H^i(\cA(D))$   is isomorphic to $H^{1-i}(\cA(\gok-D)')$, $i=0,1$,  while the character group of the latter is isomorphic to the former group. 

 \smallskip
 
 In the arithmetic case  $H^1(\cA(D))=0$.  
 For example, if $k=\Q$ then $H^1(\cA_\Q(D))=\A_\Q/(\A_\Q(D)+\Q)=0$ and $H^0(\cA_\Q(D)')=0$.
 Using the standard character $\psi_*$ of $\A_\Q$ (specified e.g. in \cite{T}, and 4.3 Ch. 3 of \cite{Fb}), the complement $\A_\Q(0)^\perp$ is  $\prod \Z_p\times\{0\}$, 
 so $H^1(\cA_\Q(0)')=\A_\Q/(\A_\Q(0)'+\Q)=\A_\Q/(\A_\Q(0)^\perp+\Q)$ equal to the quotient 
 $(\prod \Z_p \times\R+\Q)/(\prod \Z_p\times\{0\}+\Q)\simeq \R/\Z$ isomorphic to the character group of $H^0(\cA_\Q(0))=\Z$.  
This can be seen using  $(G+H)/(G+K)\simeq H/(K+G\cap H)$ and the projection on the real component.

\medskip

\noindent {\bf 1-3}.   In order (a)  to have a generally non-zero number $h^1$
 corresponding to the zero $H^1$ in the arithmetic case, and (b) not to involve the smaller adelic space $\A'$ in the arithmetic case, one can use integrals of certain functions instead of the cohomologies (i.e. linear algebra objects).  
  One can view some aspects of the classical works of Iwasawa \cite{I} and Tate \cite{T},  see also Weil \cite{W} and sect. 5 Ch. 3 of \cite{Fb}, as  contributing to (a) and (b) in dimension one. The paper  \cite{Bo} defined numbers $h^i$ in a way essentially similar to 1-5 below, but without using adelic language.     

\smallskip

 In this paper we proceed as follows.  

For an idele $\alpha=(\alpha_v)\in\A^\times$ 
define  
$$f_\alpha(u)=\prod_v f_{\alpha_v}(u_v)\colon \A\To\R, \quad u=(u_v),$$ 
where 
$f_{\alpha_v}(u_v)=f_{1}(\alpha_v^{-1}u_v)$ and 
$$
f_1(u_v)=
\begin{cases}
\text{\rm char}_{O_v}(u_v) \qquad\qquad\ \text{\rm \ if $v$ is a non-archimedean place},\\
\exp\bigl(-e_v\pi |\,u_v|_v^{2/e_v}\bigr) \quad \text{\rm \  if $v$ is an archimedean place.} 
\end{cases}
$$
Here 
$e_v$ is the $\R$-dimension 
 of the archimedean completion of $k$ with respect to $v$.  

Thus, $f_\alpha(u)=f_1(\alpha^{-1}u)$ and  $f_\alpha(0)=1$. 
In the geometric case $f_\alpha=\text{\rm char}_{\alpha\A(0)}$, in the arithmetic case the restriction of $f_\alpha$ on $\A'$ is $\text{\rm char}_{\alpha\A(0)'}$. 

The map $\A^\times\ni\alpha\To\{f_\alpha:\alpha\in \A^\times\}$  factorises through the well defined map  from replete divisors $\{D_\alpha:\alpha\in\A^\times\}$ to the set $\{f_\alpha:\alpha\in \A^\times\}$.

 It is well known that the Gaussian function can be viewed as the archimedean analogue of the characteristic function of the ring of integers (or a proper non-zero fractional ideal) of the non-archimedean completion of a global field in the sense that for almost all non-archimedean $v$ they are the eigenfunctions of appropriately normalised local Fourier transforms.

\medskip

\noindent {\bf 1-4}.   
Then in the geometric case  
$$\int_k \, f_\alpha=\int_{ k\cap \alpha \A(0)} 1=\#\, k\cap \alpha\A(0),$$ 
where the measure on $k$ and on $k\cap \alpha \A(0)$ is counting. 
The set $k\cap \alpha \A(0)$  is  equal to $H^0(\cA(D_\alpha))$.

In the arithmetic case we have the  integral 
$$\int_k \, f_\alpha=
 \int_{ k\cap \alpha \A(0)} 
 \exp\bigl(  
-\pi  \sum_v e_v |\alpha_v^{-1} \,u|_v^{2/e_v} \bigr) du=
\sum_{u\in  k\cap \alpha \A(0)} 
\exp\bigl(  
-\pi  \sum_v e_v |\alpha_v^{-1} \,u|_v^{2/e_v} \bigr),
$$
where the measure on $k$ and on $k\cap \alpha \A(0)$ is counting, the internal sum is taken over archimedean places. 

\smallskip
 
Now, both in the geometric and arithmetic cases define
\smallskip
 $$\text{\rm h}^0(D_\alpha):=\int_k \, f_\alpha.  
 $$ 
 \smallskip
This $\text{\rm h}^0(D_\alpha)$ coincides with the cardinality $\#\, H^0(\cA(D_\alpha))$ in the geometric case.
 
 \medskip
 
\noindent {\bf 1-5}.   
A real valued  function $g$ on $\A$, integrable on $k$, normalised  by the condition $g(0)=1$, induces the normalised function on $\A/k$   
$$g_{\A/k}\colon a+k\mapsto  \biggl(\int_k g (a+u) du\biggr)\,  \biggl(\int_k g(u) du\biggr)^{-1} .$$

If we apply this to $f_\alpha$, we get 
 the  integrable function  on $\A/k$:
$$f_{\alpha, \A/k}\colon a+k\mapsto  \biggl(\int_k f_\alpha (a+u) du\biggr)\,  \biggl(\int_k f_\alpha (u) du\biggr)^{-1} .$$ 
In the geometric case $f_{\alpha, \A/k}=\text{\rm char}_{(k+\A(D_\alpha))/k}$. 

Both in the geometric and arithmetic cases define $\text{\rm h}^1$  as 
\smallskip
$$\text{\rm h}^1(D_\alpha):=\biggl(\int_{\A/k}\, f_{\alpha, \A/k}\biggr)^{-1},$$ 
\smallskip
where the measure on the compact $\A/k$ is the probabilistic one.
In the geometric case this gives the usual number: 
$$\text{\rm h}^1(D_\alpha)=|\A/k: (k+\A(D_\alpha))/k|=\#\, \A/(k+\A(D_\alpha))=\#\,  H^1(\cA(D_\alpha)).$$ 
In the arithmetic case the number $\text{\rm h}^1(D_\alpha)$ is more interesting than $0=H^1(\cA(D_\alpha))$. 

\medskip 

The adelic complex $\cA(D)$ involves spaces depending on $D$. In contrast, 
the integrals defining $\text{\rm h}^j(D_\alpha)$ involve the spaces not dependent on  idele $\alpha$ (or the divisor), but the function one integrates does depend on $D_\alpha$.

 \medskip
 
\noindent {\bf 1-6}.  
Both in the geometric and arithmetic cases, the definitions imply 
$$
\text{\rm h}^1(D_\alpha)=\text{\rm h}^0(D_\alpha)\,\left( \int_{\A}\, f_\alpha \right)^{-1},
$$
where the measure on $\A$ is normalised in the following way: it is the product of the counting measure on discrete $k$ and the dual to it probabilistic measure on compact $\A/k$. Hence this translation invariant measure on $\A$ is {\em self-dual} with respect to the standard character $\psi_*$ of $\A$. It is well known that this measure   can be described as the product of the following measures on completions of $k$: the translation invariant measure on the non-archimedean completion $k_v$ of $k$ which gives the local integers $O_v$ the volume  $|O_v:\gd_v|^{-1/2}$, $\gd_v$ is the  different of $k_v$, the Lebesque measure on real numbers and twice of the Lebesque measure on  complex numbers (see e.g. \cite{T}).

 \medskip
 
\noindent {\bf 1-7}.   Denote
\smallskip
$$h^j(D_\alpha):=\log \, \text{\rm h}^j(D_\alpha), \quad \chi_{\Bbb A}(D_\alpha):=h^0(D_\alpha)-h^1(D_\alpha).$$
\smallskip

Using the previous results, we deduce \smallskip
 $$
 \chi_{\Bbb A}(D_\alpha)=  \log \int_{\A}\, f_\alpha.
$$
\smallskip

In the arithmetic case for a divisor $D=D_\alpha$ this is the same as 
the minus $\log$ of the covolume of the lattice $H^0(D)=k\cap \A(D)$ in its realification $(k\cap \A(D))\otimes_\Z\R= k\otimes_\Z\R$
when the latter   is endowed with the product of the measures on the real and complex completions of $k$ as specified in 1-6.
For a replete divisor $D_\alpha$
the number  $\chi_{\Bbb A}(D_\alpha)$ equals the minus $\log$ of the covolume of $k\cap \A(E)$, where $E$ is the divisor corresponding to the non-archimedean part of $\alpha$   
 in its realification $ k\otimes_\Z\R$ with respect to the rescaled measure 
 on the latter corresponding to the archimedean part of $\alpha$.

 \medskip

\noindent {\bf 1-8}.
Let $\gk$ be an idele such that $D_\gk$ is $\gok$ defined in 1-2.

Using 1-6, we get 
$$ \chi_{\Bbb A}(D_\alpha)= \log |\alpha| + \frac{1}{2} \, 
 \log|\gk|, 
 $$
where  $|\alpha|$ is the module of idele $\alpha$ (in the geometric case $|\alpha|=q^{-\deg D_\alpha}$)  and  
$$|\gk|=
\begin{cases}  q^{2-2g}  &\mbox{in the geometric case,} \\
  \gD_{k/\Q}^{-1} &\mbox{in the arithmetic case,} \end{cases}
$$
where $g$ is the genus of a smooth, proper, irreducible, geometrically integrable  curve over $\F_q$ with the function field $k$,  $\gD_{k/\Q}$ the absolute value of the discriminant  of the number field $k$. 

In particular, 
$$\chi_{\Bbb A}(D_\alpha)-\chi_{\Bbb A}(D_1)=\log |\alpha|.$$
In the geometric case 
$$\chi_{\Bbb A}(D_\alpha)-\chi_{\Bbb A}(D_1)=\log q \cdot \deg D_\alpha 
.$$

 \medskip
 
 \noindent {\bf 1-9}. 
 For a function $g$ on a subquotient $E$ of $\A$ call the function $g(0)^{-1} g$ its normalisation. 
 
The group of characters of $k$ is isomorphic to $\A/k$, e.g. \cite{W}.  
The  Fourier transform of a function $g$ on $k$ is the function on $\A/k$ whose value at $y+k\in \A/k$ is the integral over $k$ with respect to its counting measure of the product  $g(x)\psi_*(xy)$ with standard character $\psi_*$. 
 The definition of $f_\alpha$ and its behaviour  with respect to the Fourier transform on $\A$ associated to the character $\psi_*$ and the self-dual measure $\mu$ on $\A$ imply that the restriction of $f_{\gk \alpha^{-1}}$ on $k$ is equal to the normalisation of  the Fourier transform of $f_{\alpha,\A/k}$.  
   Hence 
 $f_{\alpha,\A/k}$ is the inverse Fourier transform of $f_{\gk \alpha^{-1}}$   times the inverse of the value of the Fourier transform of $f_{\alpha,\A/k}$  at 0.
 So $1=f_{\alpha,\A/k}(0)$ equals the product of  $\int_k f_{\gk \alpha^{-1}}$ and the Fourier transform of $f_{\alpha,\A/k}$  at 0, i.e. $\int_{\A/k} f_{\alpha,\A/k}$. Thus, $$\int_k f_{\gk \alpha^{-1}}\int_{\A/k} f_{\alpha,\A/k}=1,$$ i.e.   
$$h^1(D_\alpha)=h^0(D_{\gk\alpha^{-1}}),$$ 
and 
 $$\chi (D_\alpha)=-\chi(D_{\gk\alpha^{-1}}).$$ 

This and the last formula in 1-8 immediately imply the Riemann--Roch theorem. Of course, the argument in this subsection is closely related to the summation formula based on the adelic duality:
$$\int_k f_\alpha
=\int_k \cF(f_\alpha).
$$

 \medskip
 
\noindent {\bf 1-10}.  There is another important normalisation of the measures on  $k$ and $\A/k$ which takes into account the integral subspace $\A(0)'$:  
choose  the newly normalised measure $\mu_{\,\text{\rm new}}$ on $\A/k$  so that the volume of  
$(\A(0)'+k)/k$ is 1 
 and choose the dual to it  atomic measure on $k$. 
 Note that in order for an analogue of the argument in 1-9 to give a similar duality formula for $h^0$ and $h^1$, the measures on $\A/k$ and on $k$ should be dual to each other, i.e. 
volume of $\A/k$ times the volume of any point of $k$ should be equal to 1.
 The product of the newly normalised measures gives the same self-dual measure with respect to $\psi_*$  on $\A$ as before.

 In the geometric case the previous volume of $\A/(\A(0)+k)$ is $q\,\exp(-\chi_{\Bbb A}(D_1))=q^{\,g}$ and so the volume of 
 $(\A(0)+k)/k$ is $q^{-g}$, 
  hence $\mu_{\,\text{\rm new}}(\A/k)=q^{\,g}$  and the new atomic measure of a point of $k$ is $q^{-g}$. In the arithmetic case, similarly to 1-2, the previous volume of $\A/(\A(0)'+k)$ equals to  the volume of 
$(O_k\otimes_\Z \R)/O_k$ with the previously normalised archimedean measures, 
hence it equals $\exp(-\chi_{\Bbb A}(D_1))=|\gD_{k/\Q}|^{1/2}$, so the volume of 
 $(\A(0)'+k)/k$ is $|\gD_{k/\Q}|^{-1/2}$, hence 
$\mu_{\,\text{\rm new}}(\A/k)=|\gD_{k/\Q}|^{1/2}$
and  the new atomic measure of a point of $k$ is $|\gD_{k/\Q}|^{-1/2}$.

\smallskip 

{\em From now on we will use the notation  $\sint_{T} $ \ for \  $\log\int_T$.} 

For the newly normalised measures we have
$$
\begin{aligned}
h_{\,\text{\rm new}}^0(D_\alpha)&=\lint_k \, f_\alpha\, \mu_{\,\text{\rm new}},\\
h^1_{\,\text{\rm new}}(D_\alpha)&=- \lint_{\A/k}\, f_{\alpha, \A/k}\,\mu_{\,\text{\rm new}},
\\
 \chi_{\,\text{\rm new}}(D_\alpha)&=  \!\! \lint_{\A}\, f_\alpha \,\mu_{\,\text{\rm new}}\,= \chi(D_\alpha) .   \end{aligned}
$$

Similarly to 1-9, we get  
 $$h^j_{\,\text{\rm new}}(D_\alpha)=h^{1-j}_{\,\text{\rm new}}(D_{\gk\alpha^{-1}}),\qquad  \chi_{\,\text{\rm new}} (D_\alpha)=-\chi_{\,\text{\rm new}}(D_{\gk\alpha^{-1}}).
 $$ 
 
The function $h_{\,\text{\rm new}}$ with respect to the new normalisation of the measures
satisfies $h^j_{\,\text{\rm new}}=h^j-c$ where $c=\log \mu_{\,\text{\rm new}}(\A/k)$; in particular $c=h^1(D_1)$ in the geometric case. 

\medskip

Following the same technique, while using another natural normalisation of the measures, one can establish the relative Riemann--Roch theorem for arithmetic and geometric curves. See \cite{W} for details.

\bigskip

\noindent {\bf 2. Geometric adeles on surfaces}

\bigskip 

\noindent We continue to use the notation of 1-1.

\smallskip
 
\noindent {\bf 2-1}.  
Let $S$ be either  (1) a proper, smooth, geometrically integral surface over a finite field $\Bbb F_q$ (geometric case) 
or 
 (2)   a regular, integral, proper and flat over $B=\text{\rm Spec}(O_k)$, $k$ a number field, scheme with fibre dimension one, whose generic fibre is a smooth projective, geometrically irreducible curve $C$ over   $k$ (arithmetic case). 
 Denote by $K$ the function field of $S$.

Geometric  adeles on $S$, without the archimedean data part, were first (correctly) defined in \cite{B} and their first study was conducted by various researchers. 
In case (1)  they were studied using their topology in \cite{ARR}. 
The full definition  in  case (2) was first given in \cite{CD}, it includes the archimedean data. The definition   in \cite{CD} needs to be corrected: at real archimedean data one uses the  two-dimensional local fields with the coefficient field $\R$, see below. 

 \smallskip

In the geometric case we use the word \lq curves\rq\  for  irreducible proper  curves. In the arithmetic case  we use the word \lq curves\rq\  for irreducible proper vertical curves, irreducible regular horizontal curves together with their archimedean points, and all  \lq archimedean fibres\rq\   $S_v=S\times_B k_v$, $v$ is an archimedean place and $k_v$ is the completion of $k$. 

\smallskip 

Recall that two-dimensional local fields are complete discrete valuation fields whose residue field is a one-dimensional local field (non-archimedean or archimedean), for surveys of various results about them see  \cite{FK} and unpublished \cite{MM1}.

In this subsection we will define the objects $\bA_y^0$ and $\bA_y$ for curves $y$ on $S$, and then geometric adeles $\bA$. It will take some time.

\begin{Definition}
Let $y$ be an  \lq archimedean fibre\rq\   $y=S_v$ over $k_v$. Define $K_y$ as the function field of $S_v$. 
The field $K$ is canonically embedded in $K_y$. 
For a closed point $x$ of $y$ define $\cO_{x,y}$ as the completion of the local ring of $S_v$ at $x$ and define $K_{x,y}$ as $\cO_{x,y}\otimes K$. The field $K_y$ is canonically embedded in 
the  two-dimensional local field $K_{x,y}$.

Define the object $\bA_y$   as the usual (one-dimensional) adelic ring of the curve $S_v$ over $k_v$ and define $\bA_y^0$ as the subring of adeles with integral local coordinates. 
\end{Definition}

\begin{Definition}
Let  $y$ be a curve that is not an  \lq archimedean fibre\rq. 
Denote by  $K_y$  the fraction field of the completion $\cO_y$
of the local ring of $S$ at $y$. The field $K_y$ is a complete discrete valuation field with the ring of integers $\cO_y$, its maximal ideal $\cM_y$ and residue field $k(y)$. 
The field $K$ is canonically embedded in $K_y$. 

For a closed point $x$ of $S$  denote by $\cO_x$  the completion of the local ring at $x$. 
For a closed point $x$ of a curve $y$ consider the localisation of $\cO_x$ at the local equation of $y$ at $x$ and complete it with respect to the intersection of its maximal ideals, denote the result by $\cO_{x,y}$ and 
let $K_{x,y}$ be its quotient ring. The ring $\cO_{x,y}$ (resp. $K_{x,y}$) is isomorphic to 
the product of all $\cO_{x,z}$ (resp. $K_{x,z}$) where $z$ runs through all   minimal prime ideals of the completion of the local ring of $y$ at $x$, 
 i.e. through all formal branches  of $y$ at $x$. Denote by $\cM_{x,y}$ the intersection of maximal ideals of $\cO_{x,y}$.

For a horizontal curve $y$ 
the field $k(y)$ is canonically embedded in $\cO_y$. 
Let $x$ be a closed point of $y$, with a local branch $z$ of $y$ at $x$, or let $x$ be an archimedean point of $y$ (then $z=y$), in both cases we get a place $w$ of $k(y)$.
Define $\cO_{x,z}$ as the projective limit of $\cO_y/\cM_y^r \otimes_{k(y)} k(y)_w$, $r>0$,  where $k(y)_w$ is the completion of $k(y)$. Define  $K_{x,z}=K_{w,y}$ as $\cO_{x,z}\otimes K$. When $x$ is a closed point, this produces the objects canonically isomorphic to the objects above. 

 In all these cases $K_y$ is canonically embedded in $K_{x,z}$. 
 The fields $K_{x,z}$, $K_{w,y}$ are  two-dimensional local fields.  
\end{Definition}

\smallskip

\begin{Definition}
Define  $\bA_y^0$   as the {subring} of $\prod_{x\in y} \cO_{x,y}$, including archimedean points $x$ on horizontal curves $y$ in the arithmetic case, such that for every positive integer $r$ for almost all closed points $x$ of  $y$  the $(x,y)$-component is in $\cO_x+\cM_{x,y}^r$.  
Define the adelic object $\bA_y$  associated to the curve $y$ as the minimal subring of the product of $K_{x,y}$, for all points $x$ (archimedean included) of $y$, containing $\bA_y^0$ and $K$. 
\end{Definition}
    
    \smallskip
    
Then  {the ring} $\bA_y$ is the direct limit of $\bA_y^r=\cM_y^r\bA_y^0$, $r\in\Z$.  

Similarly to sect. 1 of \cite{ARR}, in the case when $k(y)$ and $K$ are of the same characteristic, the definitions imply that $\bA_y\simeq \A_{k(y)}((t_y))$ where $t_y$ is a local parameter of $\cO_y$. 

\smallskip

\begin{Definition}
The {ring of} adeles $\bA$  is the restricted  product of the rings $\bA_y$ for all curves $y$ with respect to their closed subrings $\bA_y^0$, i.e. the additive group of  $\bA$  is the restricted product of the {additive groups of} $\bA_y$ for all curves $y$ with respect to their closed  subgroups $\bA_y^0$ and the multiplicative group of  $\bA$ is the restricted product of {the multiplicative groups of } $\bA_y$ for all curves $y$ with respect to their closed subgroups of invertible elements of $\bA_y^0$.
\end{Definition}

Thus, the {ring}  
$\bA$ is a {subring} of the product  $\prod K_{x,y}$ of all two-dimensional local rings $K_{x,y}$, and it can be viewed as specified by the  adelic restricted product conditions described above.

\begin{Definition}
For a closed point $x$ of $S$ denote by $K_x$ be the minimal subring  of $\bA$ which contains $K$ and $\cO_x$, it is a subring of $K_{x,y}$ for any curve $y$ passing through $x$. The ring $K$ is canonically embedded in $K_x$ and  $K_x$ is canonically embedded in $K_{x,y}$. 
\end{Definition}
$$
\xymatrix@!0{
 & &  & {K_{x,y}}\ar@{-}[dl] \ar@{-}[dr]  &
\\
  & & {K_y}\ar@{-}[dr]  & & {K_x}\ar@{-}[dl]  
\\
  & & & K  & 
}
$$

 \medskip
 
\noindent {\bf 2-2}.  
We now define the topological structures. Similarly to the one-dimensional adeles, topological considerations play an important role in the study of two-dimensional adeles. 

\smallskip

  Recall that a linear topology on a  group is a  translation invariant topology in which the identity element of the group has a basis consisting of some  subgroups. For a group with a topology, the strongest linear topology weaker than the original topology is called its linearisation. 
  
  The topologies of the groups in section 1 are linear. For  groups related to two-dimensional adeles one needs to work with linearised topologies. 
  \smallskip
  
 The induced topology on a subgroup from a linear topology on a group is linear. The quotient topology on the quotient group of a group with linear topology is linear. 
  It is well know that the product topology  of groups endowed with linear topologies is linear and the inverse limit topology of groups endowed with linear topologies is linear. 
  
  Recall that an abelian topological group $G$ is the linear direct  limit of abelian topological groups $G_i$ if it is their direct limit algebraically, its topology is linear and their preimages of its open subgroups in each $G_i$ are open subgroups of $G_i$. In other words, the topology of the linear direct limit is the linearisation of the topology of the direct limit. The linear direct  limit topology on groups showing up in the two-dimensional theory can be not equivalent to the direct limit topology, see e.g. Example following the second definition in 6.2 Part I of \cite{FK}. 

\begin{Definition}The topology of the restricted product $\prod'_{i\in I} G_i$ of linear topological groups $G_i$ with respect to their closed subgroups $H_i$ is defined as the linear direct limit topology of  $G=\varinjlim G_J$, where $G_J=\prod_{i\in J} G_i \prod_{i\not\in J} H_i$ are endowed with the product topology (hence linear), and $J$ runs through  finite subsets of $I$. 
\end{Definition}

\begin{Definition}
For a closed point $x$ of $S$ define the  topology of $\cO_x$ as the $\cM_x$-adic topology, where $\cM_x$ is the maximal ideal of $\cO_x$. This is a linear topology and we also get the linear topology of the $\cM_x$-module $a\cO_{x}$ and of 
 $(a\cO_{x}+\cM_{x,z}^m)/\cM_{x,z}^m$, $a\in \cO_{x,z}$, where $\cM_{x,z}$ is the maximal ideal of $\cO_{x,z}$.  Define the linear topology of 
$\cO_{x,z}/\cM_{x,z}^m$ as the linear direct limit topology of  $(a\cO_{x}+\cM_{x,z}^m)/\cM_{x,z}^m$ with $a$ running through all elements of  $\cO_{x,z}$.
Define the  topology of $\cO_{x,z}$ as the inverse limit topology of the linear topologies of  $\cO_{x,z}/\cM_{x,z}^m$, it is linear. Define the canonical topology of $K_{x,z}$ as the linear direct limit of the linear topologies of $b\cO_{x,z}$, $b$ runs through all elements of $K_{x,z}$.
\end{Definition}

\begin{Definition}
Suppose that $k(y)$ and $K$ are of the same characteristic. Then one can  view $\cO_y/\cM_y^r \otimes_{k(y)} k(y)_w$ is a finite dimensional vector space over the topological field $k(y)_w$, so it has its canonical  topology and it is linear.  Define the topology of $\cO_{x,z}$ as the inverse limit topology of the linear topologies of  $\cO_y/\cM_y^r \otimes_{k(y)} k(y)_w$, and the topology of $K_{x,z}$ is the linear direct limit topology of the linear topologies of $b\cO_{x,z}$, $b$ runs through all elements of $K_{x,z}$. 
\end{Definition}
   
The latter definitions in the equal characteristic case, unlike the definitions before them, also work for  archimedean points $x$ of $y$, thus defining the canonical topologies of the archimedean (viewed two-dimensionally) objects $\cO_{x,z}$, $K_{x,z}$, $K_{w,y}$. 

When $x$ is a closed point of $S$, it is an easy verification that the topologies in the latter definition coincide with  the previously defined topologies.

In the equal characteristic case these topologies are equivalent to the topologies defined in  \cite{ARR}.

For a variety of different topologies on higher local fields see \cite{St}. 

The canonical two-dimensional topology on the additive group of two-dimensional local fields induces the usual one-dimensio\-nal topology on the additive group of their residue fields.

\smallskip

\begin{Definition} Define the canonical  topology of the additive groups of $\bA_y^r$ and $\bA_y$ as the induced  topology  from the   product topology of the canonical linear topologies of $K_{x,y}$. It is linear.  
\end{Definition}
    
\smallskip

The definitions imply that $\bA_y^r$ is   isomorphic  to the  inverse limit  of linear topological groups   $\bA_y^r/\bA_y^{r+s}$, $s>0$, endowed with the  quotient topology, 
and that  the canonical topology of $\bA_y$ is the linear direct limit of the topologies of $\bA_y^r$.

For an   \lq archimedean fibre\rq\   $y$ the two-dimensional topology of $\bA_y$ takes into account the topology of the archimedean completion of $k$, so in particular it is {\em different from} the one-dimensional adelic topology of $\bA_y$.

\smallskip 

\begin{Definition}Define  the canonical topology of the additive group $\bA$ as  the induced   topology from the  product topology of the additive groups of  $\bA_y$. It is linear. 
\end{Definition}

\begin{Definition}
Endow closed subgroups of the additive group $\bA$ with its canonical topology, with the induced  topology; it is linear. 
One immediately checks that the latter  on the subgroups on which the topology was already defined is equivalent to it.
Endow the multiplicative group of invertible elements $E^\times$ of a  subring $E$ of $\bA$ with the induced topology
from the product $E\times E$ via $E^\times \To E\times E$, $\alpha\mapsto (\alpha,\alpha^{-1})$.  
\end{Definition}

One easily checks that 
in positive characteristic the topology of  $\bA$ is equivalent to the topology defined in \cite{ARR}.

\medskip
 
\noindent {\bf 2-3}.  
The {ring of} adeles $\bA$ has two important subobjects: $\bB$ associated to all curves on $S$ and $\bC$  associated to all  closed points and archimedean points on $S$.

\begin{Definition}
Define the subobject $\bB$  as the intersection of the product $\prod K_y$ of the fields $K_y$, for all curves $y$ as in 2-1, with $\bA$ inside the product $\prod K_{x,y}$ of all two-dimensional local rings $K_{x,y}$. 
\end{Definition}

\begin{Definition}
To define subobject $\bC$ in the arithmetic case, we need, in addition to the definition of $K_x$ for a closed point $x$ of $S$ in 2-1,  to define the ring $K_x$ for  a closed point $x$  of an  \lq archimedean fibre\rq\   $S_v$. Consider two cases. If $x$ is algebraic over $k$ (i.e. $k(x)$ is a finite extension of $k$) then   
 $x$ lies on the unique  horizontal curve  $y$ containing $x$. 
The field  $K_{x,S_v}$ 
 is a complete discrete valuation field whose residue field and the coefficient subfield is  also the coefficient subfield of $K_{x,y}$,
while a local parameter of $S_v$ at $x$ serves as a local parameter of the complete discrete valuation field $K_{x,y}$.
Thus,  the complete discrete valuation fields $K_{x,S_v}$ and $K_{x,y}$ are canonically isomorphic. 
Define $K_x$  as the diagonal in the product of $K_{x,y}$ and $K_{x,S_v}$. 
If  $x$ is a transcendental point of $S_v$ (i.e. $k(x)$ is not an algebraic extension of $k$)
 define $K_x$ as the ring of integers of $K_{x,S_v}$ with respect to its discrete valuation. 
 
Define the adelic subobject $\bC$ as the intersection of the product  $\prod K_x$ of the rings $K_x$, for all closed points of $S$ and closed points of \lq archimedean fibres\rq, with $\bA$ inside the product $\prod K_{x,y}$ of all two-dimensional local rings $K_{x,y}$. 
\end{Definition}

This definition  corrects the definition in \cite{CD} of the subobject $\bC$ in the archimedean case.  

\smallskip
 
The function field $K$ of $S$ is canonically diagonally embedded in  $\bB$ and in  $\bC$ and in $\bA$.  For the  inclusion in $\bC$  note that the image of $K$ is in the ring of integers of $K_{x,S_v}$
 for transcendental points $x$. 
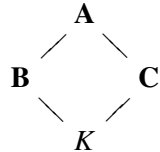
 $$
\xymatrix@!0{
 & &  & {\bf A}\ar@{-}[dl] \ar@{-}[dr]  &
\\
  & & {\bf B}\ar@{-}[dr]  & & {\bf C}\ar@{-}[dl]  
\\
  & & & K  & 
}
$$

One defines similarly rational adeles $\bA^{rt}$ whose $(x,y), (x,S_v),(w,y)$-components are in $K$. The rational adeles are dense in $\bA$.

\smallskip

\begin{Definition}
Define the subobject $\bA(0)$ as the intersection of $\bA$ with the product of $\bA_y^0$ for all curves that are not  \lq archimedean fibres\rq\   and of $\bA_y$ for  \lq archimedean fibres\rq\   in the arithmetic case. 
\end{Definition}

\begin{Definition}
Define $\bB_y=K_y$. Define $\bB_y^r$ as the intersection of $\bB_y$ with $\bA_y^r$ inside $\bA_y$. 
\end{Definition}

 \medskip
 
\noindent {\bf 2-4}.  We now define a non-trivial character $\psi$ of $\bA$, its continuity follows from the definitions.

\begin{Definition}

In the geometric case a non-zero $\omega\in \Omega^2_{K/\F_q}$ produces the residue map $\bA\to \F_q$, \cite{ARR}. Compose it with the trace to $\F_p$ and an injective homomorphism from $\Bbb F_p$ to $p$th roots of unity in the complex unit circle $S^1$ in order to get a character $\psi$ of $\bA$. In the arithmetic case for a non-zero $\omega\in \Omega^1_{K/k}$  use the  associated residue map $\bA\to\A_k$ composed with  a non-trivial character $\psi_0\colon\A_k\to S^1$ vanishing on $k$  to get a character $\psi$ of $\bA$, sect. 28 of \cite{AoA2}, \cite{MM2}, \cite{CD}.  Recall that at closed points $x$ of an  \lq archimedean fibre\rq\   $S_v$ one adds the minus sign in front of the residue map for the local object associated to $x$ on $S_v$,  \cite{CD}. 
\end{Definition}

The  character $\psi$ of $\bA$ vanishes on $\bB$ and $\bC$, due to the additive reciprocity law for curves and points of $S$, \cite{AoA2}, \cite{MM2}, \cite{CD}. 

Using the  character $\psi$ of $\bA$  we get  the  pairing 
$$\langle\,\,,\,\,\rangle=\langle\,\,,\,\,\rangle_\psi\colon\bA\times\bA\To S^1, \quad (a,b)\mapsto \psi(ab).$$

 This pairing is continuous and non-degenerate, see \cite{ARR} and \cite{CD}.

\smallskip

 Using this pairing and similar pairings for $\bA_y$, $K_{x,y}$, it was established in \cite{ARR} (its argument has an obvious extension to the arithmetic case as well) and  \cite{CD}, that the {\em additive topological group of $K_{x,y}$, $\bA_y$, $\bA$   is  isomorphic to the group  of its characters endowed with the  compact-to-open topology. In other words, those groups are algebraically and topologically self-dual.}

 \smallskip
 
 \noindent {\bf 2-5}.  
 The paper  \cite{Ka} established an extension of Pontryagin duality from a category of locally compact abelian groups to  the categories of countable direct limits of locally compact abelian groups and of countable inverse limits of locally compact abelian groups. Similarly to (3) of Theorem in sect. 2 of \cite{ARR}, we have

\begin{prp} 
For every  subgroup $B$ of $\bA$ its complement $B^\bot$ with respect to $\langle\,\,,\,\,\rangle$  is a closed subgroup.
For every closed subgroup $B$ of $\bA$ we have $(B^\bot)^\bot=B$
and $B$ is topologically isomorphic  to the group of characters of $\bA/B^\bot$. 
For two subgroups  $B,C$ we have  $(B + C)^\bot=B^\bot\cap C^\bot$.
For two closed subgroups $B,C$ we have $(B\cap C)^\bot=B^\bot+C^\bot$ and $B+C$ is closed.
For two closed subgroups $B\supset C$ the group $X(B/C)$ of characters of  $B/C$  is  topologically isomorphic to $X(C^\bot/B^\bot)$.
\end{prp}

\begin{proof} The definition of $\bB'$ implies that it is isomorphic to the countable inverse limit of $\bB'/\bB(D)'$
 with respect to all divisors $D$. It will be established in 2-9, without using this proposition and its use, that $\bA'/\bB'$ is isomorphic to the countable  direct limit of compact $(\bA(D)'+\bB')/\bB'$ with respect to all divisors $D$. 
Hence, by  Th.1 and Th.2 of \cite{Ka},   characters of these groups separate closed subgroups and points outside them,  and characters of these groups extend from closed subgroups to those groups. Hence $\bA'$ has the same properties, and hence $\bA$ has the same properties. 
Then the arguments of Prop. 8, Prop. 11 and Prop. 12 of \cite{BHM} imply the property in second sentence, which then implies the rest of the properties.
\end{proof}

\smallskip

Papers \cite{MM2} and  \cite{CD} established  the strong additive reciprocity 

\begin{thm} 
 The (orthogonal) complement $\bB^\perp$ of $\bB$ with respect to the pairing $\langle\,\,,\,\,\rangle$ coincides with $\bB$, and the  complement $\bC^\perp$ of $\bC$ with respect to the pairing $\langle\,\,,\,\,\rangle$ coincides with $\bC$. 
 \end{thm}
 
 The argument in \cite{CD}  works for  the subobject $\bC$ defined in this paper. Indeed, 
since one chooses a non-zero differential form in $\Omega_{K/k}$ to define the character $\psi$ of $\bA$, the conductor of the  character $\psi_{x ,S_v}$ at every transcendental point $x$ is the ring of integers of $K_{x,S_v}$, since polynomials over $K$ do not have zeros in transcendental over $K$ points; hence the  complement of the ring of integers of $K_{x,S_v}$ with respect to $\psi_{x, ,S_v}$ equals  it; thus, the argument of  \cite{CD} implies that the  complement of $\bC$ with respect to $\langle\,\,,\,\,\rangle$ equals $\bC$. 

\smallskip

The previous theorem implies that $\bB$ and $\bC$, as the orthogonal complements, are closed in  $\bA$.

 \medskip
 
\noindent {\bf 2-6}.  
The homomorphism from the characters of the group $\bA/(\bB+\bC)$ to the  characters of  
$\bA^{rt}/(\bB\cap \bA^{rt}+\bC\cap\bA^{rt})$, induced by $\bA^{rt}\to \bA$, is injective since its kernel is closed and the closure of $\bA^{rt}$ is  $\bA$. 
Theorem of sect. 2 of \cite{ARR} shows that $\bB\cap\bC=K$ in the geometric case. 
Exactly the same argument as in its proof works in the arithmetic case. Hence 

\begin{prp}
$\bB\cap\bC=K$ both in the geometric and  arithmetic cases. 
\end{prp}

\smallskip

Using Theorem 2-4  and similarly to the proof of Theorem of sect. 2 of \cite{ARR} we deduce

\begin{crl}
The complement $K^\perp$ of $K$ with respect to $\langle\,\,,\,\,\rangle$ is the closed group $\bB+\bC$.
The  group of characters of  $K$ is isomorphic to $\bA/(\bB+\bC)$ both in the geometric and  arithmetic cases. 
\end{crl}

 \medskip
 
\noindent {\bf 2-7}.  For an archimedean place $v$ let $k_v'$ be the algebraic closure of $k_v$. Denote $S_v'=(S_v\times_{k}
 \Bbb C)(\Bbb C)$ where the product is taken with respect to the embedding of $k$ in $\Bbb C$ corresponding to $v$. The morphism $S_v\To S_v'$ induces the morphism $\bA_{S_v}\To \bA_{S_v'}$.

\begin{Definition}
Let  $y=S_v$. Let $\mu_y$ be the canonical (1,1)-form on the Riemann surface $S_v'$, \cite{dJ}.
 Then the $\mu_y$-volume of the surface is 1.

Let $|\,\,|_v'$ be the unique extension of the normalised absolute value $|\,\,|_v$  on $k_v$, defined in 1-1, to $k_v'$ such that its restriction on $k_v$ coincides with $|\,\,|_v$.

Introduce
  $$\iota=\iota_y\colon  \bA_y \To [0,+\infty)\subset \Bbb R, \quad  
\iota_y(\alpha_y)=\exp\bigl(\int_{S_v'} \log | 
\alpha_y 
(z) |_v' \, \,\mu_y(z) \bigr), \quad 0\mapsto 0,
$$
where $\alpha_y(z)$ is the value (i.e. its image modulo the maximal ideal of  $\cO_{z,S_v'}$) of the $z$-component of the adele $\alpha_y$ viewed as an element of $\bA_{S_v'}$ at the point $z\in S_v'$, with the exception of finitely many points where it does not belong to $\cO_{z,S_v'}$ and its value is not determined. The value $\iota_y(\alpha_y)$ is defined when the function $ \log | \alpha_y (z) |_v' $ of $z$ is integrable against $\mu_y$. The domain of $\iota_y$ is a proper subset of $\bA_y$. This domain includes $\bB_y=k_v(S_v)$. The image of $\iota_y$ is $[0,+\infty)$. 
\end{Definition}

Then $\iota_y(\alpha_y\beta_y)=\iota_y(\alpha_y) \iota_y(\beta_y)$ when the factors on the right are defined.

 \medskip

\noindent {\bf 2-8}. Discrete subgroups of divisors can be viewed as surjective images of appropriate subgroups of locally compact (non-discrete)  multiplicative adeles. 

\begin{Definition}
For an idele $\alpha=(\alpha_y)\in \bB^\times$ define the {\em replete 
divisor} $$D_\alpha=-\sum v_y(\alpha_y) [y]$$ of $S$, where $y$ runs through all  curves in the geometric case and through  all curves in the arithmetic case as in 2-1. Here $v_y$ is the discrete (surjective to $\vZ$) valuation of $\bB_y$ associated to a local parameter $t_y$ of curve $y$ and  $$v_y(\alpha_y):=v(\iota_y(\alpha_y))=-\log|\iota_y(\alpha_y)|_v=- e_v \int_{S_v'} \log | 
\alpha_y 
(z) |_v' \, \,\mu_y(z) $$ 
for the   \lq archimedean fibre\rq\   $y=S_v$  over an archimedean place $v$, $e_v$ was defined in 1-3.

Then $v_y(\alpha_y\beta_y)=v_y(\alpha_y)+v_y(\beta_y)$. 
\end{Definition}

The zero divisor is $D_1$. 

In the geometric case we get the usual divisor. 

When $\alpha\in K^\times$, the replete divisor $D_\alpha$ is  minus  the Arakelov divisor of  $\alpha$.

The maps $$\bB^\times\ni\alpha\longmapsto D_\alpha$$  are surjective homomorphisms   to the group of divisors in the geometric case  
and to the group of replete divisors in the arithmetic case. 
 The substantial difference with the one-dimensional case is that it is not the group of invertible elements of the full ring $\bA$, but of the ring $\bB$, which is used to lift divisors to the topological adelic level. 

\smallskip

 
Below, except 4-6,  $D$ will  denote a usual divisor, while the notation $D_\alpha$ will stand for a replete divisor (which in the geometric case is the usual divisor).

\smallskip

\begin{Definition}
Similarly to 1-1, 
for a  divisor $D=\sum n_y [y]$ define the subgroup $$\bA(D)=\{\beta\in\bA: v_y(\beta)\geq-n_y\},$$ where $y$ runs through all curves in the geometric case and through all horizontal and vertical curves in the arithmetic case.  
In particular,  $\bA(0)=\bA(D_1)$. 
\end{Definition}

This definition can be extended to replete divisors by ignoring their   \lq archimedean fibres\rq\    components; then 
for a replete divisor $D_\alpha$ we have    $\bA(D_\alpha)=\alpha\bA(0)$. 

\smallskip

The topological group  $\bA$ is the  direct  limit of its closed subgroups $\bA(D)$ where $D$ runs over all divisors.

\smallskip

\begin{Definition}
Denote by $\bA'$ the subobject of $\bA$ of adeles which have  zero components on   \lq archimedean fibres\rq. 
 In the geometric case  $\bA'=\bA$. 
 
 In the arithmetic case denote by $\bA^{af}$ the part of $\bA$ for  \lq archimedean fibres\rq\   and by $\bA^{st}$ the part of $\bA$ on curves outside  \lq archimedean fibres\rq. In the geometric case $\bA^{af}$ is empty and $\bA^{st}=\bA$. Then $$\bA=\bA^{st}\times \bA^{af}, \quad \bA'=\bA^{st}\times 0_{\bA^{af}}.$$

 \smallskip

 For any subset $G$ of $\bA$ define $$G(D):=G\cap \bA(D), \quad G':=G\cap \bA'.$$ So $G'(D)=G(D)'$. 
  \end{Definition}
 \smallskip
 
Thus, $\bB'$ is isomorphic to $\bB^{st}$, the part of $\bB$ based outside  \lq archimedean fibres\rq. 
The diagonal image $K^{st}$ of $K$ in  $\bA^{st}$ is isomorphic to $K$.  
The object $\bC'$ coincides with $\bC$ at non-archimedean points and it is zero at  archimedean points of horizontal curves and at closed points of  \lq archimedean fibres\rq.  

\smallskip

 In the arithmetic case from 2-6 we deduce  $\bB'\cap \bC'=K'=0$.

\smallskip

 \begin{Definition}
 Let $\gK$ be the divisor such that 
 the  complement $\bA(0)^\perp$ of $\bA(0)$ with respect to $\langle\,\,,\,\,\rangle_\psi$  is $\bA(\gK)'$. 
  \end{Definition}
In the geometric case the divisor $\gK$ is  linearly equivalent to the canonical divisor. 
In the arithmetic case it {does not correspond} to  the canonical sheaf  of $S$ over $B$:   the definition of the pairing in 2-4 involves the non-trivial character $\psi_0$, so the sheaf associated to $\gK$ is different from the algebraic geometric relative dualising sheaf of $S$ over $B$, since the latter does not take into account the full two-dimensional geometric-arithmetic situation involving the arithmetic character $\psi_0$. 

\smallskip

 Then $$\bA(D)^\perp=\bA(D^\circ)', \quad  \text{\rm where $D^\circ=\gK-D$}. 
 $$
 
 Similarly to the proof of Theorem of sect. 2 of \cite{ARR}, the results of 2-4 imply that 
 $$\bB(D)'+\bC(D)'=(\bB(D^\circ)\cap \bC(D^\circ))^\perp$$ and hence it is closed.

 \medskip 
 
\noindent {\bf 2-9}.
For every curve $y$ which is not an  \lq archimedean fibre\rq\   we have isomorphisms $$\bA_y^r/\bA_y^{r+1}\simeq \A_{k(y)}, \quad \bB_y^r/\bB_y^{r+1}\simeq {k(y)}.$$
 Since $\A_{k(y)}/k(y)$ is compact, the group 
$\bA_y^r/\bB_y^r$ is the  inverse limit of compact groups and hence is compact. 
For an  \lq archimedean fibre\rq\   $y=S_v$ over an archimedean place $v$  the quotient $\bA_y/\bB_y=\A_{\,k_v(S_v)}/k_v(S_v)$ is $k_v$-linearly compact in its one-dimensional adelic topology, see e.g. sect. 0 of \cite{ARR}.   Therefore we deduce

\begin{lem}
The  quotient $\bA(D)'/\bB(D)'$   is compact.
The quotient  $(\bB(D)'+\bC(D)')/\bB(D)'$ is a closed 
subgroup of compact $\bA(D)'/\bB(D)'$ and hence compact and so is $\bC(D)'/K(D)'\simeq (\bB(D)'+\bC(D)')/\bB(D)'$.
\end{lem}

Hence $\bA'/\bB'$ is the  direct limit of compact $(\bA(D)'+\bB')/\bB'$.

\begin{crl}
$\bC(D)'$ is  compact.
\end{crl}
 \begin{proof}
In the arithmetic case $\bC(D)' = \bC(D)'/K(D)'$ is compact and $K'=0$. 
In the geometric case $K(D)$ is finite, see \cite{ARR}. 
 \end{proof}

\begin{thm}
 $K$ is discrete in $\bA$. The group of its characters, isomorphic to  $\bA/(\bB+\bC)$, is  compact.
 \end{thm}
 \begin{proof}
 In the geometric case see \cite{ARR}.  The discreteness of  $K$ in $\bA$ follows from is by the discreteness of 
$K\cap \bA(D)$ in $\bA$ for every divisor $D$.  In the arithmetic case 
 $K\cap \bA(D)$ is a finitely generated $O_k$-module: indeed, $K\cap \bA(D)$  is isomorphic to the Zariski $H^0(D)$, \cite{B}, \cite{MM2}. 
The latter is a finitely generated $O_k$-module, see e.g. Th. 5.2 Ch. III of \cite{RH}.  Consider $K\cap \bA(D)$ as a subobject of  $K\otimes_{\Q}\R=\prod \bB_y\subset \prod \bA_y$ (with respect to the diagonal embedding), the product is taken over all  \lq archimedean fibres\rq. The group 
$K\cap \bA(D)$ is a discrete subset of  a finite dimensional $\R$-space $(K\cap \bA(D))\otimes_{\Q}\R$ with its topology induced from the topology of $\prod \bA_y$. Hence $K\cap \bA(D)$ is discrete in $\bA$.
 \end{proof}

 \medskip
 
\noindent {\bf 2-10}.  
For a divisor $D$, a
 two-dimensional analogue $\cA_S(D)$ of the one-dimensional adelic complex $\cA(D)$ is
 \smallskip
$$ K\oplus \bB(D)\oplus \bC(D)\To \bB\oplus \bC\oplus \bA(D)\To \bA, $$
$(a_0,a_1,a_2)\mapsto (a_0-a_1,a_2-a_0, a_1-a_2)$ and $(a_{01},a_{02}, a_{12})\mapsto a_{01}+a_{02}+a_{12}$.

\medskip

The complex $\cA_S(D)$ is quasi-isomorphic to the complex 
$$\bC(D)\To \bA(D)/\bB(D)\To \bA/(\bB+\bC),$$
the maps from $\cA_S(D)$ are  given by the 
  minus projection to the third term, the projection to the third term and the quotient, the quotient; for the geometric case see \cite{ARR}.

The cohomology objects are  

 $$\begin{aligned}
 H^0(\cA_S(D))&\simeq K\cap \bA(D),\\ 
H^1(\cA_S(D))&\simeq (\bB+\bC)(D)/(\bB(D)+\bC(D)),\\
H^2(\cA_S(D))&=\bA/(\bB+\bC+\bA(D)).\end{aligned}$$ 

\smallskip

The group $H^0(\cA_S(D))$ is discrete and the group $H^2(\cA_S(D))$ is  compact by 2-9. 
We have  isomorphisms $H^1(\cA_S(D))\simeq ((\bB+\bC)\cap (\bB+\bA(D)))/(\bB+\bC(D)) \simeq
((\bB+\bA(D))\cap 
 (\bC+\bA(D))/(K+\bA(D))$, 
see \cite{ARR} for the geometric case, the arithmetic case is similar.

It is shown in \cite{ARR} that the adelic cohomology spaces are of finite dimension over $\Bbb F_q$ in the geometric case, so they are finite. 
They are not finite in general in the arithmetic case, for example $H^0(\cA_S(0))=O_k$.  

\medskip

For a divisor $D$ one can also   consider the following  complex $\cA_S'(D)$: 
\smallskip 
$$ K\oplus \bB(D)'\oplus \bC(D)'\To \bB\oplus \bC\oplus \bA(D)'\To \bA.$$

\smallskip

Similarly to the above, 
the complex $\cA_S'(D)$ is quasi-isomorphic to the complex 
$$\bC'(D)\To \bA'(D)/\bB'(D)\To \bA/(\bB+\bC).$$

We have

$$\begin{aligned}  H^0(\cA_S'(D))&\simeq K\cap \bA'(D)\quad\text{ (which is zero in the arithmetic case)}, \\
H^1(\cA_S'(D))&\simeq ((\bB+\bC)\cap \bA'(D))/(\bB'(D)+\bC'(D)),\\
H^2(\cA_S'(D))&=\bA/(\bB+\bC+\bA'(D)).\end{aligned}
$$

\begin{prp} 
 The character group of 
$H^j(\cA_S(D))$ 
 is isomorphic to  $H^{2-j}(\cA_{S}'(D^\circ))$ for $j=0,1,2$.
 \end{prp} 
 
 This follows from using the   continuous and non-degenerate pairing $\langle\,\,,\,\,\rangle$, Theorem in 2-4 and the description of the complement of $\bA(D)$ in 2-8.  
  See \cite{ARR} in the geometric case. In the arithmetic case 
 use arguments similar to \cite{ARR}. For example, the group of characters of $H^1(\cA_S'(D))$
  is isomorphic to $((\bB+\bA(D^\circ))\cap 
 (\bC+\bA(D^\circ))/(K+\bA(D^\circ))$
 isomorphic to  $H^1(\cA_S(D^\circ))$.

 \smallskip 
We deduce 
 
 \begin{crl} 
$H^2(\cA_S(D))=\bA/(\bB+\bC+\bA(D))$ is  finite in the geometric case and zero in the arithmetic case. 
 \end{crl}

Lemma 2-9 implies that $H^1(\cA_S'(D))$ is  compact.  
So in the geometric case, $H^1(\cA_S(D))$ is simultaneously  compact and discrete by the previous proposition, hence finite.

 \smallskip
 
\noindent {\bf 2-11}.  
Recall that  $\bA$, $\bA'$ and $\bA_y$ are not  locally compact  groups. Otherwise  two-dimensional local fields associated to $S$ would have been  locally compact but they are not. 
For example, every open subgroup (and hence closed) of the additive group of a two-dimensional local field $F$  is not  compact. Indeed, in the equal characteristic case, such a subgroup contains a subgroup of the form $V=\sum_{-\infty}^{+\infty} U_i t_2^i\cap F$ with $U_i$ being open subgroups of the one-dimensional residue field and $U_i\subset U_{i+1}$ for all $i$, however for $V$ to be compact, $U_i$ should be 0 for almost all negative $i$ since otherwise $V$ has a cover  by open $V_j=\sum V_{i,j} t_2^i\cap F$, where $V_{i,j}=U_i$ for all $i\ge -j$ and $V_{i,j}$ is a properly smaller open subgroup of $U_i$ for all $i<-j$, 
which does not have a finite subcover. One argues similarly in the mixed characteristic case. 

The space $F$ is  not a Baire space: it can be represented as the union of  a countable collection of compact subsets with empty interior  (of type $\sum_{i\geq i_0} K_i t_2^i$ with compact $K_i$).

  \smallskip

The quotient  $\bA/\bB$ 
 is not  locally compact. Indeed, otherwise  its character group isomorphic to $\bB$ would be  locally compact and then $\bA$ would be  locally compact, a contradiction. Similarly, $\bA/\bC$ is not  locally compact. Since $\bC/K$ is homeomorphic  to $\bA/\bB$ and $\bB/K$ is homeomorphic to $\bA/\bC$, they are not  locally compact and neither are $\bB$ and $\bC$. 
  
 Since $\bA_y$ is not  locally compact, the similar argument shows that $\bA_y/\bB_y$ is not  locally compact,  at the same time  the quotient $\bA_y/\bB_y$  is the  direct limit of compact spaces $\bA_y^r/\bB_y^r$.

\bigskip

\noindent {\bf 3. Selective integration on two-dimensional geometric adeles} 

\bigskip

\noindent {\bf 3-1}.  
 We will define  finitely additive translation invariant measures on certain  subquotients of $\bA$ important for applications.  They are either direct limits of  compact subquotients or inverse limits of discrete subquotients.
 Normalisation aspects will be important, but still allowing a degree of flexibility. 
 
 One can develop 
a more general theory of translation invariant measures on subquotients of $\bA$, but since this general theory is not needed in this paper, we do not include it.

Some of the finitely additive translation invariant measures defined below are tensor product of the relevant measures on all curves $y$. However, note that none of them is the tensor product of measures on (two-dimensional)  local fields associated to $x\in y$, unlike the one-dimensional measure on $\Bbb A$ in 1-6 and the measure on two-dimensional analytic adeles in \cite{AoA2}. 

\smallskip

First, we will define the normalised  translation invariant finitely additive measures $\mu_{\bA'/\bB'}$,  $\mu_{\,\bB'}$, $\mu_{\bA'}$ on
objects $\bA'/\bB', \bB', \bA'$.  

\smallskip

\begin{Definition}
Recall that $\bA'/\bB'$ is the  direct limit of compact subquotients $(\bA(D)'+\bB')/\bB'$, with $D$ running through all divisors. This subquotient is isomorphic to 
$(\bA(D)'+\bB)/\bB$ and to $\bA(D)'/\bB(D)'$,  
Define the set of measurable subsets of $\bA'/\bB'$ as the ring of sets generated by measurable subsets of  compact $(\bA(D)'+\bB')/\bB'$ for all $D$ (so each such measurable subset  is a measurable subset of  some compact $(\bA(D)'+\bB')/\bB'$). 
In order to define a finitely additive translation invariant measure on $\bA'/\bB'$, we need to fix the normalisations of the Haar volumes $c_D$  of compact $(\bA(D)'+\bB')/\bB'$ so that they respect the direct limit with respect to $D$. 
For $D=0$ choose a non-trivial Haar measure on compact $(\bA(0)'+\bB')/\bB'$ so that its volume is $$c_0=c_0( \bA',\bB')$$ is 1. 
If $E$ is the divisor of a curve $y$ then the quotient $(\bA(D+E)'+\bB')/(\bA(D)'+\bB')$ is canonically isomorphic to the quotient $\bA(D+E)'/(\bA(D)'+\bB(D+E)' )$ non-canonically  isomorphic to compact 
$\A_{k(y)}/k(y)$ with its normalised volume $m_E$ equal to 1,  as in 1-6. 
 Using this as a step to pass from one divisor to another, define  the volume $c_D$ of compact $(\bA(D)'+\bB')/\bB'$  so that the relations $c_{D+E}=c_{D} m_E$ hold for all divisors $D$.
 Hence  $$c_D=c_0\prod m_{D_i}^{r_i},  \quad \text{\rm where $D=\sum r_i D_i$}$$
  and $D_i$ are divisors of curves.  
 Due to the choice of $c_0$ and $m_{D_i}$, we have $c_D=1$. 
Due to this formula, the measure of a measurable subset of $(\bA(D)'+\bB')/\bB'$ as a subset of $(\bA(D)'+\bB')/\bB'$ is equal to its measure as a subset of $(\bA(D+E)'+\bB')/\bB'$ for any effective divisor $E$. 
Now define the measure $\mu_{\bA'/\bB'}$ of a measurable subset of $\bA'/\bB'$ as its normalised measure in any compact $(\bA(D)'+\bB')/\bB'$  which it is a subset of. 
This a translation invariant finitely additive measure $\mu_{\bA'/\bB'}$.  
\end{Definition}

\begin{Definition}
The group of characters of compact $(\bA(D)'+\bB)/\bB$ is the discrete space $\bB/\bB(D^\circ)\simeq \bB'/\bB(D^\circ)'$, while  $\bB'$ is  the  inverse limit of the latter with respect to all divisors $D$.  Define the translation invariant finitely additive measure $\mu_{\,\bB'}$ on $\bB'$ by duality to the definitions in the previous paragraph. In other words, similarly to  1-10,  the measures on the compact subquotient group and its dual group with respect to the character $\psi$ from 2-4 should satisfy the property: 
the volume of the compact group  times the volume of any point of its dual group is 1.
So 
the normalisation of the volume of $(\bA(D)'+\bB')/\bB'$ corresponds to the atomic measure on  $\bB'/\bB(D^\circ)'$ in which every point has volume  $c_D^{-1}$, which is 1 in our choice.  
\end{Definition}

\smallskip

\begin{Definition}
Now define the translation invariant finitely additive measure $\mu_{\bA'}$ on $\bA'$ as the tensor product of the defined measures on $\bA'/\bB'$ and on $\bB'$. Similarly to the first paragraph of 1-10 this measure does not depend on the choice of positive $c_0$, and it is self-dual with respect to $\psi$. 

\smallskip

Similarly, we define the normalised translation invariant finitely additive measures $\mu_{\bA_y/\bB_y}$,
 $\mu_{\,\bB_y}$, $\mu_{\bA_y}$ 
for every curve $y$. The definitions imply that $\mu_{\bA'/\bB'}$, $\mu_{\,\bB'}$, $\mu_{\bA'}$ are their tensor products.
\end{Definition}

\smallskip

Next, we will define the  normalised translation invariant measures $\mu_{(\bB+\bC)'/\bB'}$, $\mu_{(\bB+\bC)'}$,  $\mu_{\,\bB'/\Delta K}$, $\mu_{\bA'/(\bB+\bC)'}$,  $\mu_K$, $\mu_{\bA/(\bB+\bC)}$ 
 on objects 
$(\bB+\bC)'/\bB'$, $(\bB+\bC)'$, $\bB'/\Delta K$, $\bA'/(\bB+\bC)'$,  $K$, $\bA/(\bB+\bC)$,
where $\Delta K$ is the isomorphic image of $K$ in $\bA'$.  
 The general procedure is the same as above. Below in this subsection for the quotient object $Q/R$  which is the direct limits of compact subquotients
 $(Q(D)+R)/R$ we normalise the normalised measure on $(Q(0)+R)/R$ by asking that  its volume $c_0(Q,R)$ is 1,
 and we will take the atomic measure on the group of characters of $Q/R$ in which gives every point has volume $c_0(Q,R)^{-1}$.

\begin{Definition}
The quotient $(\bB+\bC)'/\bB'$ of $\bA'/\bB'$ is the  direct limit of subobjects  
$ (\bB+\bC)(D)'/\bB(D)'\simeq  ((\bB+\bC)(D)'+\bB')/\bB'$
  of  the compact object  $(\bA(D)'+\bB')/\bB'$, with their induced measure.  
Similarly to the previous material define the translation invariant finitely additive measure on $(\bB+\bC)'/\bB'$ using the normalised measures on $((\bB+\bC)(D)'+\bB')/\bB'$. 
Define the normalised measure on $(\bB+\bC)'$ as the tensor product of this normalised measure on $(\bB+\bC)'/\bB'$ and the normalised measure on $\bB'$.

\smallskip
 
The group of characters of $(\bB+\bC)'/\bB'$ is  isomorphic to  $(\bB+\bA^\ast)/(K+\bA^\ast)$, 
where $\bA^\ast=0\times \bA^{af}$. The latter quotient object is isomorphic to 
$ ((\bB+\bA^\ast)/\bA^\ast)/(( K+\bA^\ast)/\bA^\ast)
 \simeq \bB'/\Delta K$. The space of characters of $((\bB+\bC)(D)'+\bB)/\bB$
is isomorphic to $\bB/(\bB\cap (K+\bA (D^\circ))$. This way  the quotient $\bB'/\Delta K$ is 
 the  inverse limit of  discrete spaces $\bB/(\bB\cap (K+\bA (D^\circ))$. Similarly to the previous, define the normalised measure  $\mu_{\,\bB'/\Delta K}$ on  $\bB'/\Delta K$ by duality.

\smallskip

The  quotient object $\bA'/(\bB+\bC)'$ of $\bA'/\bB'$ is the  direct limit of compact  $((\bB+\bC)'+\bA(D)')/(\bB+\bC)'\simeq \bA(D)'/(\bB+\bC)(D)'$, the latter are quotients of  $\bA(D)' /\bB(D)'$. 
 Normalise the measure of $\bA(D)'/(\bB+\bC)(D)'$ so that its tensor product with the already normalised measure of $(\bB+\bC)(D)'/\bB(D)'$ is the already normalised measure of $\bA(D)'/\bB(D)'$. 
Thus we get  the translation invariant finitely additive measure $\mu_{\bA'/(\bB+\bC)'}$ on $\bA'/(\bB+\bC)'$, similarly to the previous material.  

\smallskip

The space of characters of $\bA'/(\bB+\bC)'$ is isomorphic to $(K+\bA^\ast)/\bA^\ast 
\simeq K$. 
Similarly to the previous, define the translation invariant measure $\mu_{K}$  on discrete $K$ by duality, using the normalised measures on $((\bB+\bC)'+\bA(D)')/(\bB+\bC)'$. It is an atomic measure.

\smallskip

The  quotient $\bA/(\bB+\bC)$ is isomorphic to the group of characters of $K$, define $\mu_{\bA/(\bB+\bC)}$ on $\bA/(\bB+\bC)$ by duality to the atomic measure $\mu_K$ on discrete space $K$. 
\end{Definition}

\smallskip

  The definitions imply

\begin{lem} 
$$\begin{aligned}
\mu_{\,\bA'}&=\mu_{\,\bA'/\bB'}\otimes\mu_{\,\bB'}=
\mu_{\bA'/(\bB+\bC)'}\otimes\mu_{(\bB+\bC)'},\\
\mu_{(\bB+\bC)'}&= \mu_{(\bB+\bC)'/\bB'}\otimes \mu_{\,\bB'},\\
\mu_{\bA'/\bB'}&=\mu_{\bA'/(\bB+\bC)'}\otimes\mu_{(\bB+\bC)'/\bB'}, \\
\mu_{\,\bB'}&=\mu_{\,\bB'/\Delta K}\otimes\mu_{K}.
\end{aligned}
$$
\end{lem}

\medskip

The space of functions which we will integrate  against the defined measures will be spanned by the characteristic functions of divisors and their products with characters.

\smallskip

\noindent {\bf 3-2}. 
The current understanding of the archimedean fibres part of the developing theory leads to the following

\begin{Definition}
For  an  \lq archimedean fibre\rq\  $y=S_v$ using  the surjective partially defined map $\iota_y$ from 2-7, define
$$
\mu_{\bA_y}:=\iota_y^* \mu_{\Bbb R},
$$
where $\mu_{\Bbb R}$ 
 is the Lebesque measure on real numbers. 

The space of integrable functions is $g\circ \iota_y$ where $g$ are integrable functions on $k_v$.
We have 
$$\int_{\bA_y} g(\iota_y \beta) \, \mu_{\bA_y}=
\int_{\Bbb R} 
g(x)\,\mu_{\Bbb R}. 
$$

Choose a fundamental domain in $\bA_y$ for $\bA_y/\bB_y=\Bbb A_{k_v(y)}/k_v(y)$ and define $\mu_{\bA_y/\bB_y}$
as the restriction of $\mu_{\bA_y}$ on it. 

Define  $\mu_{\bB_y}$ so that the property $\mu_{\bA_y}=\mu_{\bA_y/\bB_y}\otimes \mu_{\bB_y}$ holds. 
\end{Definition}

\begin{Definition}
Define Fourier transform $\cF_{\bA_y}$ in a similar way, using $\iota_y^*$, i.e.
$$\cF_{\bA_y}(g\circ \iota_y):=\cF_{k_v}(g)\circ \iota_y,$$
\end{Definition}

\begin{Definition}
Extend the previous definitions from one  \lq archimedean fibre\rq\  to all  \lq archimedean fibres\rq. This supplies the  measure and integration on $\bA^{af}$, $\bB^{af}$ and $\bA^{af}/\bB^{af}$, with the relevant properties,  including $\mu_{\bA^{af}}=\mu_{\bA^{af}/\bB^{af}}\otimes \mu_{\,\bB^{af}}$.
\end{Definition}

\medskip

\noindent {\bf 3-3}.  
The strong approximation theorem and the definition of $\bC$ imply
 $\bA=\bA'+\bB+\bC$. 
 Hence the embedding $\bA'\to \bA$ induces the isomorphism $\bA'/(\bB+\bC)'\simeq \bA/(\bB+\bC)$.
 Hence $\bA'/(\bB+\bC)'$ is compact and its group of characters $\Delta K$ is discrete in $\bA'$.

 We also deduce that $(\bB+\bC)/(\bB+\bC)'\simeq \bA/\bA'\simeq \bA^{af}$. 
 The quotient of $\bA/\bB$ by $(\bA'+\bB)/\bB\simeq \bA'/\bB'$ is isomorphic to $\bA/(\bB+\bA')\simeq \bA^{af}/\bB^{af}$, 
 and the quotient of $(\bB+\bC)/\bB$ by $((\bB+\bC)'+\bB)/\bB\simeq (\bB+\bC)'/\bB'$ is isomorphic to $(\bB+\bC+\bA')/(\bB+\bA')= \bA/(\bB+\bA')\simeq \bA^{af}/\bB^{af}$.

 \begin{Definition}
 Define, using 3-1 and the preceding material of  this section, the normalised translation invariant measures 
 $$\begin{aligned}
\mu_{\,\bA}:&=\mu_{\,\bA'}\otimes\mu_{\,\bA^{af}},\\
\mu_{\,\bB}:&=\mu_{\,\bB'}\otimes\mu_{\,\bB^{af}},\\
\mu_{\,\bA/\bB}:&=\mu_{\,\bA'/\bB'}\otimes\mu_{\,\bA^{af}/\bB^{af}},\\
\mu_{\,\bB+\bC}:&= \mu_{(\bB+\bC)'}\otimes \mu_{\,\bA^{af}},\\
\mu_{(\bB+\bC)/\bB}:&=\mu_{(\bB+\bC)'/\bB'}\otimes \mu_{\,\bA^{af}/\bB^{af}}.
\end{aligned}
$$
\smallskip

Using  $\bC(D)/K(D)\simeq (\bB(D)+\bC(D))/\bB(D)\subset \bA(D)/\bB(D)$ and $\mu_{\bA/\bB}$, 
define the normalised translation invariant measures 
$\mu_{\,\bC/K}$. Then define $\mu_{\,\bC}:=\mu_{\bC/K}\otimes \mu_K$. 
\smallskip

Define the normalised translation invariant measure $\mu_{\,\bB/K}$ so that the formula $\mu_{\,\bB}=\mu_{\, \bB/K}\otimes \mu_K$ holds.
\smallskip

Define the normalised translation invariant measure  $\mu_{(\bB+\bC)/\bC}$ so that the formula $\mu_{\,\bB+\bC}=\mu_{(\bB+\bC)/\bC}\otimes \mu_{\,\bC}$ holds.

\smallskip

Define the normalised translation invariant measure $\mu_{(\bB+\bC)/K}$ so that the formula $\mu_{\,\bB+\bC}=\mu_{(\bB+\bC)/K}\otimes \mu_{\,K}$ holds.
\end{Definition}

\smallskip

 Using 3-1 we deduce

\smallskip

\begin{lem} 
$$\mu_{\bA}=\mu_{\bA/\bB}\otimes\mu_{\,\bB}, \quad  \mu_{(\bB+\bC)/K}=\mu_{(\bB+\bC)/\bB}\otimes \mu_{\,\bB/K}=\mu_{(\bB+\bC)/\bC}\otimes \mu_{\,\bC/K}. $$
\end{lem}

\medskip

\noindent {\bf 3-4}.  
Similarly to 1-5,  we have the following

\begin{Definition}
Let  $R\subset Q$ be subquotient objects of $\bA$. Suppose that the normalised measure on $R$ is defined. 
For a real valued  function $g$ on $Q$, integrable on $R$, normalised  by the condition $g(0)=1$, 
define the normalised function
$$g_{Q/R}\colon Q/R\To \R, \quad  q+R\mapsto  \biggl(\int_R g (q+r) \,\mu_R(r)\biggr)\,  \biggl(\int_R g(r)\, \mu_R(r)\biggr)^{-1} .$$
Note the abuse of notation: this function may depend on $R$, not just on the quotient $Q/R$. 

 \end{Definition}

 When $R=0$ and $\mu_R$ is an atomic measure,  we will use the notation $g$ for $g_{Q/0}$, the  restriction of $g$ to $Q$. 
 
 \medskip

From now on we will use the notation  $\sint_{T} $ \ for \  $\log\int_T$.

\smallskip

From the definitions and an extension of the Fubini property from integrals over compact and discrete groups to their limits, we deduce

\begin{lem} 
Let 
 $R\subset Q$ be two subquotient objects of $\bA$, on which the normalised measure is defined.
 Let  
 for a real valued  function $g$ on $Q$, normalised  by the condition $g(0)=1$, 
the two integrals on the right are  finite. Then 
$$  \lint_{Q/R} g_{Q/R}=\lint_{ Q} g - \lint_{R} g_{R} .
$$
\end{lem}

\bigskip
\medskip

\noindent {\bf 4. Numbers $h^i(D)$ is dimension two, and the Euler characteristic} 

\bigskip

\noindent {\bf 4-1}.  
We will be working with the following functions on adeles and their quotient subobjects.

\begin{Definition}
For an idele $\alpha=(\alpha_y)\in \bB^\times$  define the function $$f_\alpha(u)=\prod_y f_{\alpha_y}(u_y)\colon \bA\To\R,\qquad u=(u_y), u_y\in \bA_y,$$
 where $f_{\alpha_y}(u_y)=f_{1}(\alpha_y^{-1}u_y)$ and 
 
 $$
 f_{1}(u_y)=
 \begin{cases}
 \text{\rm char}_{\bA_y^0}(u_y) \quad\  \text{\rm \  for non-archimedean and horizontal curves $y$},\\
 \exp\bigl(-e_v\pi\, |\iota_y(u_y)|_v^{2/e_v}\bigr) \quad \qquad \, \text{\rm   for an \lq archimedean fibre\rq\  $y=S_v$  over an  archimedean place $v$,}
\end{cases}
$$
\medskip
$e_v$ is defined in 1-3, $\iota_y$ is defined in 2-7.  
\end{Definition}
   
Compare with 1-3.

The map $\bB^\times\ni\alpha\To\{f_\alpha:\alpha\in \bB^\times\}$  factorises through the well defined map  from divisors/replete divisors $\{D_\alpha:\alpha\in\bB^\times\}$ to the set $\{f_\alpha:\alpha\in \bB^\times\}$.

\smallskip

 The function $f_\alpha$ is normalised: $f_\alpha(0)=1$. 
 
 \smallskip 

 The restriction of $f_\alpha$ on $\bA'$  is $\text{\rm char}_{\alpha\bA(0)'}$. 
 
 Following the definitions, for an \lq archimedean fibre\rq\ $y$ and $\alpha_y\in \bB_y^\times$  we have
 $$
 \lint_{\bA_y} f_{\alpha_y} \, \mu_{\bA_y}=\lint_{\bA_y} f_{1} \, \mu_{\bA_y}+ \log \iota_y(\alpha_y)= \log\iota_y(\alpha_y)=-v_y(\alpha_y).
 $$

 \smallskip
 
 For subobjects $R\subset Q$ of $\bA'$ we get 
 $$f_{\alpha, Q/R}=\text{\rm char}_{(R+Q(D_\alpha))/R}=\text{\rm char}_{(R+\alpha\bA(0)\cap Q)/R}$$ if $\mu_R( R\cap\alpha\bA(0))$ is finite positive.

 \medskip
 
\noindent {\bf 4-2}.  
The definitions in 3-1 and 2-9 imply that the integrals $\sint_{\bA/(\bB+\bC)}\, f_{\alpha,\bA/(\bB+\bC)}\, \mu_{\bA/(\bB+\bC)}$, $\sint_{\bA/\bB}\, f_{\alpha,\bA/\bB}\, \mu_{\bA/\bB}$, 
$\sint_{( \bB+\bC)/\bB}\, f_{\alpha, (\bB+\bC)/\bB}\,\mu_{ (\bB+\bC)/\bB}$, $ \sint_{\bC/K}\, f_{\alpha,\bC/K}\,\mu_{\,\bC/K}$, $\sint_{K} \, f_\alpha \, \mu_K$,  $\sint_{\bC} \, f_\alpha \, \mu_{\bC}$ are finite, while $\sint_{\bA} \, f_\alpha \, \mu_{\bA}$ is not. 

In the geometric case we also have 
$$\sint_{\bC} \, f_\alpha \, \mu_{\bC}=\log \mu_{\bC}(\bC(D_\alpha)), \quad 
\sint_{\bA/\bB}\, f_{\alpha,\bA/\bB}\, \mu_{\bA/\bB}=\log \mu_{\bA/\bB} ( \bA(D_\alpha)/\bB(D_\alpha)).
$$

\smallskip

Keeping in mind the cohomology spaces in 2-10 and extending 1-10, introduce

\begin{Definition}
For $\alpha\in\bB^\times$ define
\smallskip
 $$
 \begin{aligned}
 h^0(D_\alpha)&:=\lint_{K} \, f_\alpha \, \mu_K, \\
 h^2(D_\alpha)&:=-\lint_{\bA/(\bB+\bC)}\, f_{\alpha,\bA/(\bB+\bC)}\, \mu_{\bA/(\bB+\bC)}. 
 \end{aligned}
 $$ 
 This integrals are defined in view of the properties of the adelic objects discussed in sect. 2. 
 \end{Definition}

 \smallskip
Note the minus sign for $h^2$, compare with the minus sign in the definition of  $h^1$ in 1-6, 1-7.

\smallskip

Denote $$c_\star:=\mu_K(\{0\}) >0.$$ 

In the geometric case we obtain, 
$$\begin{aligned}
h^0(D_\alpha)= \lint_K \, f_\alpha\, \mu_K =\log \mu_K(\{0\}) \,+ \, \log\#\, K\cap \alpha\bA(0)&=\log c_\star \,+ \, \log\#\, K\cap \alpha\bA(0)
\\ 
& =\log c_\star \,+ \, \log\#\, H^0(\cA_S(D_\alpha)).
\end{aligned}
$$

\smallskip

In the arithmetic case, $K\cap \alpha \bA(0)$ is a finitely generated $o_k$-module.  Similarly to 1-4, we obtain  
$$
\begin{aligned}
h^0(D_\alpha)=\lint_K \, f_\alpha\, \mu_K &=
\log \int_{ K\cap \alpha \bA(0)} 
 \exp\bigl(  
-\pi  \sum_v e_v |\iota_{S_v}(\alpha_{S_v}^{-1} \,u)|_v^{2/e_v} \bigr) \mu_K(u)\\
&=
\log \,c_\star+\log\, \bigl(\,\sum_{  u\in K\cap \alpha \bA(0)} 
\exp\bigl(  
-\pi  \sum_v e_v |\iota_{S_v}(\alpha_{S_v}^{-1} \,u)|_v^{2/e_v} \bigr)\bigr),
\end{aligned}
$$
  the internal sum is taken over archimedean places $v$.  Compare with the one-dimensional formula in 1-4. 
  This formula for $h^0$ agrees with the proposed formula in \cite{GS}.

\medskip
 
For a divisor $D$ we have 
$
H^1(\cA_S(D))\simeq (\bB+\bC)(D)/(\bB(D)+\bC(D))$.  It is 
  isomorphic to the quotient of  
 $(\bB+\bC)(D)/\bB(D)$ by $(\bB(D)+\bC(D))/\bB(D)$, and the quotient 
 $(\bB(D)+\bC(D))/\bB(D)$ is  isomorphic to $\bC(D)/K(D)$.

\begin{Definition}
For $\alpha\in\bB^\times$ define

\smallskip
 $$
 h^1(D_\alpha):=\lint_{( \bB+\bC)/\bB}\, f_{\alpha, (\bB+\bC)/\bB}\,\mu_{ (\bB+\bC)/\bB} - \lint_{\bC/K}\, f_{\alpha,\bC/K}\,\mu_{\,\bC/K}.
 $$
\end{Definition}

\smallskip

 Even though $( \bB+\bC)/\bB\simeq \bC/K$, the functions $f_{\alpha, (\bB+\bC)/\bB}$ and $f_{\alpha,\bC/K}$ are different, and 
  $ h^1(D_\alpha)$  is in generally non-zero; see about the abuse of notation in 3-4.  
  
  \smallskip
  
  \begin{erthm}
 
  In the geometric case, comparing with 2-10, we immediately deduce that 
  $$h^j(D_\alpha)=\log\#\,  H^j(\cA_S(D_\alpha))+n_j$$ for an appropriate constant $n_j$.
   In particular, $n_0=\log c_\star$.
\end{erthm}
 
  \medskip
 
\noindent {\bf 4-3}.  
 Now, for the Euler characteristic 
 we get
 \smallskip
 $$
 \begin{aligned}
 \chi_{\bA}(D_\alpha)= \chi (D_\alpha)&=h^0(D_\alpha)-h^1(D_\alpha)+h^2(D_\alpha)\\
 &=-\lint_{(\bB+\bC)/\bB}\, f_{\alpha,  (\bB+\bC)/\bB}\,\mu_{( \bB+\bC)/\bB}-\lint_{ \bA/(\bB+\bC)}\, f_{\alpha,\bA/(\bB+\bC)}\,\mu_{\bA/(\bB+\bC)} \\
 &\phantom{=\ \ }+\lint_{K} \, f_\alpha\,\mu_K+ \lint_{\,\bC/K}\, f_{\alpha, \bC/K}\,\mu_{\,\bC/K}. 
  \end{aligned}
 $$
 \smallskip

Using  Lemma 3-3 and  3-4 we obtain

\begin{thm}
 $$
 \chi (D_\alpha)=\lint_{ \bC}\, f_{\alpha}\,\mu_{\,\bC}- \lint_{\bA/\bB}\, f_{\alpha,\bA/\bB}\,\mu_{\,\bA/\bB}.
 $$
 \end{thm}
 
 \smallskip
 
 Compare with the formula in 1-7.

 \medskip
 
 \noindent{\bf 4-4}. 
Using the character $\psi$ of $\bA$  and the  pairing $\langle\,\,,\,\,\rangle$ in 2-4,
 we can use the natural extension of  harmonic analysis on abelian locally compact groups, to work with Fourier transforms of functions on direct limits of compact subquotients and dual inverse limits of discrete subquotients.

 Let $\kappa\in\bB^\times$ be such that the divisor $\gK$ defined in 2-8 is equal to $D_\kappa$.
 
Let $E$ be a discrete subquotient of $\bA'$ and let $G$ be a compact subquotient of $\bA$ isomorphic to its character space, compatible with the pairing $\langle\,\,,\,\,\rangle$ of 2-4. Let the normalised measures $\mu_G$ and $\mu_E$ be defined and they are related to each other in the dual way, similarly to 3-1.  Similarly to 1-9, we deduce from the 
 definitions  that $f_{\kappa\alpha^{-1},E}$ on $E$ is equal to the normalisation (i.e. the value at 0 is 1) of  the Fourier transform $\cF(f_{\alpha, G})$  of $f_{\alpha, G}$.
 Hence $f_{\alpha,G}$ is the inverse Fourier transform $\cF'(f_{\kappa\alpha^{-1},E})$ of $f_{\kappa\alpha^{-1},E}$ 
  times  the value $\cF(f_{\alpha,G})(0)$ of the Fourier transform of $f_{\alpha,G}$  at 0.
 So $$1=f_{\alpha,G}(0)=\cF'(f_{\kappa\alpha^{-1},E})(0)\cF(f_{\alpha,G})(0), $$
 with the first factor on the right hand side equal to $\int_E f_{\kappa\alpha^{-1},E}\,\mu_E$ and the second factor equal to 
 $\int_{G} f_{\alpha,G}\,\mu_G$. 
   Thus, $$\int_E f_{\kappa\alpha^{-1},E}\,\mu_E=\biggl(\int_{G}\, f_{\alpha,G}\,\mu_G\biggr)^{-1}.$$
     This formula immediately extends to the direct limits of compact and the inverse limits of discrete subquotients of $\bA'$. 
 It also extends to the archimedean fibres, since the Fourier transform there is defined using the Fourier transform of functions of real numbers and  $\iota_y$ is multiplicative.

 \medskip

 \begin{prp}
$$
h^2(D_\alpha)=h^0(D_{\kappa\alpha^{-1}}), \quad h^1(D_\alpha)=h^{1}(D_{\kappa\alpha^{-1}}).
$$
\end{prp}
\begin{proof}
To deduce the first assertion,  apply the previous formula to $E=K$ and $G=\bA/(\bB+\bC)$. 

The deduce the second assertion, use the preceding displayed formula and 
 Lemma 3-3 to obtain $$
\begin{aligned}
&\lint_{(\bB+\bC)/\bB}\, f_{\alpha,  (\bB+\bC)/\bB}\,\mu_{\, (\bB+\bC)/\bB}- \lint_{\bC/K}\, f_{\alpha, \bC/K}\,\mu_{\,\bC/K}\\&\qquad=
- \lint_{\bB/K}\, f_{\kappa\alpha^{-1}, \bB/K}\,\mu_{\,\bB/K}+\lint_{(\bB+\bC)/\bC}\, f_{\kappa\alpha^{-1},  (\bB+\bC)/\bC}\,\mu_{\, (\bB+\bC)/\bC}\\&\qquad=
- \lint_{\bC/K}\, f_{\kappa\alpha^{-1}, \bC/K}\,\mu_{\,\bC/K}+\lint_{(\bB+\bC)/\bB}\, f_{\kappa\alpha^{-1},  (\bB+\bC)/\bB}\,\mu_{\, (\bB+\bC)/\bB}.
\end{aligned}
$$
\end{proof}

\smallskip

The previous Proposition immediately implies 

\begin{thm}
$$\chi(D_\alpha)=\chi(D_{\kappa\alpha^{-1}}).$$
\end{thm}

 \smallskip

\begin{erthm}
In the geometric case, from the normalisation of the measures in 3-1 and 3-3 we get    $h^1(D_1)=h^2(D_1)=0$, hence
$$h^j(D_\alpha)=\log\#\,  H^j(\cA_S(D_\alpha))-\log\#\,  H^j(\cA_S(D_1)), \quad j=1,2.$$
So, using the notation of Remark 4-2,
$$n_0=\log c_\star, \quad n_1=-\log\#\,  H^1(\cA_S(D_1)), \quad n_2=-\log\#\,  H^2(\cA_S(D_1)).$$
Then  the  Proposition implies $h^0(D_\alpha)=\log\#\,  H^0(\cA_S(D_\alpha))-\log\#\,  H^0(\cA_S(\gK))$. 
In particular,
 the constant $c_\star$  satisfies
$c_\star^{-1}=\#\,  H^0(\cA_S(\gK))$. Using Proposition 2.10,  $n_0=n_2$.
So, $\chi(D)$ equals  $\log q$ times  the (dimension over $\Bbb F_q$) Euler characteristic of  the complex $\cA_S(D)$ plus  the constant $n_1+2n_0$. 
 \end{erthm}

  \medskip
  
  \noindent {\bf 4-5}.  We can now define the intersection index using the adelic Euler characteristic of the previous subsection.

\begin{Definition}
  For $\alpha,\beta\in \bB^\times$  define the  index
  $$
[D_\alpha,D_\beta]:=\chi(D_1)-\chi(D_\alpha)-\chi(D_\beta)+\chi(D_{\alpha\beta}).
$$
This also gives an induced pairing for replete divisors. 
\end{Definition}

\smallskip

 It is  a  symmetric form. 

For $\gamma\in K^\times= \bB^\times\cap\bC^\times$ we deduce, using translation invariance of the integration, that  
 $$ 
 \lint_{\bC}\, f_{\alpha\gamma}\,\mu_{\,\bC}=\lint_{\bC}\, f_{\alpha}\,\mu_{\,\bC}, \quad 
  \lint_{\bA/\bB}\, f_{\alpha\gamma,\bA/\bB}\, \,\mu_{\,\bA/\bB}=\lint_{\bA/\bB}\, f_{\alpha,\bA/\bB}\, \,\mu_{\,\bA/\bB},
  $$
  hence $[D_\alpha,D_\beta]=[D_{\alpha\gamma},D_\beta]$, so the pairing for replete divisors is invariant with respect to translation by principal divisors.  
  
  \begin{thm}
  In the geometric case, for a divisor $D$ whose support does not include an irreducible proper curve $y$ we have 
 $$
 \chi_{\bA} (D)-\chi_{\bA} (D-D_y)=\chi_{\Bbb A_{k(y)}}(D\vert_y)-\log m_{D_y},$$
 the numbers $m$ were defined in 3-1. 
 
 Hence 
 $$[D,D_y]=\chi_{\Bbb A_{k(y)}}(D\vert_y)-\chi_{\Bbb A_{k(y)}}(0)=\log q\cdot \deg_{\Bbb A_{k(y)}}(D\vert_y).$$
 
 The index $[\,\,,\,\,]$ is $\log q$ times with the intersection pairing.
  \end{thm}
  \begin{proof}
  For $D=D_\alpha$, $D_y=D_\beta$ and $f=f_\alpha$, $g=f_{\alpha\beta^{-1}}$, we have
  $$
  \chi_{\bA} (D)-\chi_{\bA} (D-D_y)
  =
  \lint_{\bC}\, f\,\mu_{\,\bC}-  \lint_{\bC}\, g\,\mu_{\,\bC}
  -\lint_{\bA/\bB}\, f_{\bA/\bB}\,\mu_{\,\bA/\bB}+ \lint_{\bA/\bB}\, g_{\bA/\bB}\,\mu_{\,\bA/\bB}.
  $$
  Note that the differences of the first two and of the last two integrals do not depend on the normalisation of $\mu_{\bC}$ and of $\mu_{\,\bA/\bB}$. By 2-9, $\bC(D)$ and $\bA(D)/\bB(D)$ are compact.
  Using 4-2, we obtain 
   $$ \begin{aligned}
   \chi_{\bA} (D)-\chi_{\bA} (D-D_y)
  &=\log \mu_{\bC}(\bC(D))-\log \mu_{\bC}(\bC(D-D_y))\\
  &-(\log\mu_{\,\bA/\bB}(\bA(D)/\bB(D))-\log\mu_{\,\bA/\bB}(\bA(D-D_y)/\bB(D-D_y))).
  \end{aligned}$$ 
Then 
  $$
    \chi_{\bA} (D)-\chi_{\bA} (D-D_y)
  =\log \int_{\Bbb A_{k(y)}} f_{D|_y}-\log \mu(  \Bbb A_{k(y)}/k(y))=
  \chi_{\Bbb A_{k(y)}} (D\vert_y)-
  \log m_{D_y},
  $$
  due to 
 the isomorphism  $\bC(D)/\bC(D-D_y)\simeq \Bbb A_{k(y)} (D\vert_y)$ established in Lemma in sect. 4 of \cite{ARR} and the isomorphism $(\bA(D)+\bB)/ (\bA(D-D_y)+\bB)\simeq \Bbb A_{k(y)}/k(y)$, and using 1-7 and 3-1.

 Using the properties stated before the theorem and arguing similar to the proof of 
 Th. 4 of \cite{ARR} (this argument works fine with $m_{D_y}$ not necessarily equal to 1), one deduces that 
 that the index $[\,\,,\,\,]$ is $\log q$ times  the intersection pairing.  
 
 Note that the normalisation  $m_{D_y}=1$  in 3-1 is not needed for the last two statements in the theorem.
 \end{proof}
  
    \medskip
  
  \noindent{\bf 4-6}. We end with more remarks and open questions.  
  \smallskip

\begin{erthm}
In the arithmetic case, one similarly establishes  
a formula relating $[D,D_y]$ and $\deg_{\Bbb A_{k(y)}}(D\vert_y)$ when $D$ is a  divisor and $y$ is a vertical irreducible curve on $S$.
An open problem is the compute the index  for a replete divisor $D$ and a horizontal irreducible curve $y$, the scheme theoretical closure of  a rational point of  $S_\eta(l)$ where $l$ is a finite extension of the number field $k$, with a morphism $s\colon \text{\rm Spec}(O_l)\To S$  and   a replete divisor $D\vert_y=s_y^*(D)$ of $\text{\rm Spec}(O_l)$. 

 \end{erthm}

\begin{erthm}
In the arithmetic case note the very substantial difference of the role of $\kappa$ and the behaviour of $h^i\mapsto h^{2-i}$   from the role of the canonical sheaf and behaviour of the usual cohomology.     The Arakelov adjunction formula (e.g. as in \cite{Ft} or  3.3 Ch. V of \cite{Lg}) involves the canonical sheaf of $S$ over $B$. However, from the point of view of  the two-dimensional adeles  it is  more natural to involve the divisor $\gK$ defined in 2-8, since unlike the canonical sheaf it takes into account two dimensions: the geometric and arithmetic ones. An open problem is to prove an   adjunction formula involving $\gK$. 
\end{erthm}

\begin{erthm}
The one-dimensional  Riemann--Roch formula is   $\chi_{\Bbb A}(D_\alpha)-\chi_{\Bbb A}(D_1)=\log|\alpha|$ of 1-8. Theorem 4-4 and the definition of index $[\,\,\,,\,\,\,]$  immediately imply  the two-dimensional Riemann--Roch formula $$\chi_{\bA}(D_\alpha)-\chi_{\bA}(D_1)=-2^{-1} [D_\alpha,D_{\kappa\alpha^{-1}}].$$
  In the geometric case this is the Riemann--Roch theorem in sect. 4 of  \cite{ARR}. An open problem in the arithmetic case  is to obtain an adelic two-dimensional Riemann--Roch formula and  
 compare it with the  version in \cite{Ft}, based on the use of  the canonical sheaf.
\end{erthm}

 \bigskip

\bibliographystyle{acm}
\bibliography{../../Bibliography}

\begin{thebibliography}{10}

\bibitem{A} 
{\sc S. Y. Arakelov}
\newblock An intersection theory for divisors on an arithmetic surface,
\newblock{Izv. Akad. Nauk.} 38 (1974), 1179--1192.


\bibitem{B} 
{\sc A. Beilinson},
\newblock Residues and adeles,
\newblock {Funct. Anal. Appl.} 14 (1980), 34--35.


\bibitem{Bo} 
{\sc A. Borisov},
\newblock Convolution structures and arithmetic cohomology,
\newblock Comp. Math. 136(2003),  237--254.

\bibitem{BHM}
{\sc R. Brown, P.J. Higgings, S.A. Morris},
\newblock Countable products and sums of lines and circles: their closed subgroups, quotients and duality properties, \newblock Math. Proc. Cambr. Phil. Soc. 78(1975), 19--32.

\bibitem{W}
{\sc W. Czerniawska,}
\newblock Harmonic analysis approach to the relative Riemann--Roch theorem on global fields, preprint
\newblock  \url{https://arxiv.org/abs/2208.10424}. 



\bibitem{CD}
{\sc W. Czerniawska, P. Dolce,}
\newblock Adelic geometry on arithmetic surfaces II: 
completed adeles and idelic Arakelov intersection theory, 
\newblock { J. Number Theory} 211 (2020) 235--296; updated version is available from  \url{https://arxiv:1906:03745}. 

\bibitem{dJ}
{\sc R. De Jong,}
\newblock  Arakelov invariants of Riemann surfaces, Docum. Math. (2005),  311--329.

\bibitem{Ft}
{\sc G. Faltings,}
\newblock Calculus on arithmetic surfaces,  Ann. Math. (2) 119 (1984), 387--424.


\bibitem{St}
 {\sc I. Fesenko,}
 \newblock  Sequential topologies and quotients of Milnor $K$-groups of higher local fields,
 \newblock St. Petersburg Math. J. 13 (2002) 485--501; available from 
 \newblock \url{https://ivanfesenko.org/wp-content/uploads/2021/10/stqk.pdf}.
 
 \bibitem{Ada}
 {\sc I. Fesenko,}
\newblock  Adelic approach to the zeta function of arithmetic schemes in dimension two, 
\newblock Moscow Math. J. 8 (2008) 273-310; available from 
\newblock \url{https://ivanfesenko.org/wp-content/uploads/2021/10/a2.pdf}.
 
\bibitem{AoA2}
 {\sc I. Fesenko,}
\newblock  Analysis on arithmetic schemes. II, 
\newblock {J. K-theory} 5 (2010), 437--557; available from 
\newblock \url{https://ivanfesenko.org/wp-content/uploads/2021/10/ada.pdf}.


\bibitem{ARR}
 {\sc I. Fesenko,}
\newblock  Geometric adeles and the Riemann--Roch theorem for 1-cycles on surfaces, 
\newblock { Moscow Math. J.} 15 (2015), 435--453, available from
\newblock \url{https://ivanfesenko.org/wp-content/uploads/2021/10/ar.pdf}.


\bibitem{Fb}
 {\sc I. Fesenko,}
\newblock  Essential Algebraic Number Theory, 
\newblock {World Scientific Publ., Number Theory and Its Applications 18(2026), ISBN 978-981-98-2571-4}.

\bibitem{FK}
 {\sc I. Fesenko, M. Kurihara (eds.),}
 \newblock Invitation to Higher Local Fields, 
 \newblock Geometry and Topology Monographs vol. 3, 2000.


\bibitem{GS} {\sc G. van der Geer, R. Scoof},
\newblock Effectivity of Arakelov divisors and the theta divisor of a number fields,
\newblock Sel. Math. 6(2000), 377--398. 


\bibitem{RH} {\sc R. Hartshorne},
\newblock  Algebraic Geometry,
\newblock Springer 1997. 

\bibitem{I}
{\sc K. Iwasawa,}
\newblock Hecke's L-functions, Spring 1964, Princeton Univ., 
\newblock Springer 2019. 

 \bibitem{Ka}
 {\sc  S. Kaplan,}
 \newblock  Extensions of the Pontrjagin duality II: direct and inverse sequences, 
 \newblock Duke Math.J.  17(1950)  419--435.

\bibitem{Lg} {\sc S. Lang},
\newblock  Introduction to Arakelov Theory,
\newblock Springer 1988. 




\bibitem{QL} {\sc Q. Liu},
\newblock  Algebraic Geometry and Arithmetic Curves,
\newblock OUP 2002. 

\bibitem{MM1}
{\sc  M. Morrow,} 
\newblock An introduction to higher dimensional
local fields and adeles,
\newblock preprint available from \url{https://www.imo.universite-paris-saclay.fr/~matthew.morrow/}.  

\bibitem{MM2}
{\sc  M. Morrow,} 
\newblock Grothendieck's trace map for arithmetic surfaces via 
residues and higher adeles, 
\newblock { J. Algebra and Number Th.} 6-7 (2012), 1503--1536. 

\bibitem{T}
{\sc J. Tate,}
\newblock Fourier analysis in number fields
and Hecke's  zeta function, PhD thesis, Prin\-ce\-ton Univ., 1950.

\bibitem{W}
{\sc A. Weil,}
\newblock Basic Number Theory, 3rd edit., 
\newblock Springer 1973. 


\end{thebibliography}

\end{document}